\documentclass[12pt]{amsart}
\usepackage[margin=1.1in]{geometry}
\usepackage{graphicx}

\newcommand{\RR}{\mathbb{R}}
\newcommand{\length}{\mathcal{H}^1}
\newcommand{\surf}{\mathcal{H}^{n-1}}
\newcommand{\Dtop}{D_{\mathrm{top}}}
\newcommand{\Hdim}{\mathop{\mathcal{H}\!\textrm{--}\dim}}
\newcommand{\dist}{\mathop\mathrm{dist}}
\newcommand{\diam}{\mathop\mathrm{diam}}
\newcommand{\interior}{\mathop\mathrm{int}}

\newcommand{\res}{\hbox{{\vrule height .22cm}{\leaders\hrule\hskip.2cm}}}

\newcommand{\Xint}[1]{\mathchoice
    {\XXint\displaystyle\textstyle{#1}}%
    {\XXint\textstyle\scriptstyle{#1}}%
    {\XXint\scriptstyle\scriptscriptstyle{#1}}%
    {\XXint\scriptscriptstyle\scriptscriptstyle{#1}}%
    \!\int}
\newcommand{\XXint}[3]{\setbox0=\hbox{$#1{#2#3}{\int}$}
    \vcenter{\hbox{$#2#3$}}\kern-.5\wd0}
\newcommand{\dashint}{\Xint-}

\newtheorem{oldthm}{Theorem}

\newtheorem{theorem}{Theorem}[section]
\newtheorem{lemma}[theorem]{Lemma}
\newtheorem{proposition}[theorem]{Proposition}
\newtheorem{corollary}[theorem]{Corollary}
\newtheorem{conjecture}[theorem]{Conjecture}
\theoremstyle{definition}
\newtheorem{definition}[theorem]{Definition}
\newtheorem{remark}[theorem]{Remark}

\numberwithin{equation}{section} 

\begin{document}

\title[Harmonic measure and Lipschitz approximation]{Null sets of harmonic measure on NTA domains: Lipschitz approximation revisited}
\author{Matthew Badger}
\thanks{The author was partially supported by NSF grants DMS-0838212 and DMS-0856687}
\address{Department of Mathematics, University of Washington, Box 354350, Seattle, WA, 98195-4350, USA}
\email{mbadger@math.washington.edu}
\date{October 28, 2010}
\subjclass[2000]{28A75, 31A15}
\keywords{Harmonic measure, absolute continuity, big pieces of Lipschitz graphs, corkscrew condition, NTA domain, Hausdorff dimension, Wolff snowflake}
\begin{abstract} We show the David-Jerison construction of big pieces of Lipschitz graphs inside a corkscrew domain does not require surface measure be upper Ahlfors regular. Thus we can study absolute continuity of harmonic measure and surface measure on NTA domains of locally finite perimeter using Lipschitz approximations. A partial analogue of the F. and M. Riesz Theorem for simply connected planar domains is obtained for NTA domains in space. As one consequence every Wolff snowflake has infinite surface measure.
\end{abstract}
\maketitle

\section{Introduction}
\label{intro}

What are the minimal assumptions on the boundary of a domain $\Omega\subset\RR^n$ to guarantee its harmonic measure $\omega$ and surface measure $\sigma=\surf\res\partial\Omega$ have the same null sets? When $n=2$, for example, one has the classic result of F. and M. Riesz \cite{Riesz}. In the plane, a topological condition ($\partial\Omega$ is a Jordan curve) and a mild measure-theoretic condition ($\partial\Omega$ has finite length) imply harmonic measure vanishes exactly on sets of zero length.

\begin{oldthm}[F.\ and M.\ Riesz 1916] \label{FMsc} Let $\Omega\subset\RR^2$ be a simply connected domain, bounded by a Jordan curve. If $\length(\partial\Omega)<\infty$, then \begin{equation}\label{FMsceq} \omega(E)=0\Leftrightarrow \length(E)=0\quad\text{for every Borel set }E\subset\partial\Omega.\end{equation}
\end{oldthm}

If one strengthens the hypothesis $\length(\partial\Omega)<\infty$ in Theorem \ref{FMsc}, the relationship witnessed between $\omega$ and $\sigma$ is stronger than absolute continuity \cite{L}. A Jordan curve $\partial\Omega$ is called a \emph{chord-arc curve} if $\partial\Omega$ is a quasicircle and there exists a constant $C>0$ such that $\length(\Delta(Q,r))\leq Cr$ for all $Q\in\partial\Omega$ and $0<r<\diam \Omega$, where $\Delta(Q,r)=\partial\Omega\cap B(Q,r)$.

\begin{oldthm}[Lavrentiev 1936] \label{Lsc} Let $\Omega\subset\RR^2$ be a simply connected domain, bounded by a chord-arc curve. Then (1.1) holds and $\omega\in A_\infty(\sigma)$, i.e., there exist constants $0<\delta<1$ and $0<\varepsilon<1$ such that for every $\Delta=\Delta(Q,r)$, \begin{equation}\label{Lsceq}\sigma(E)\leq \delta\sigma(\Delta) \Rightarrow \omega(E)\leq \varepsilon\omega(\Delta)\quad\text{for every Borel set }E\subset\Delta.\end{equation}
\end{oldthm}

An amusing fact is that the ``one-sided" condition (\ref{Lsceq}) implies (\ref{FMsceq}). Actually $\omega\in A_\infty(\sigma)$ if and only if $\sigma\in A_\infty(\omega)$; see \cite{CF}, also for several equivalent definitions of $A_\infty$ weights. For further discussion on harmonic measure in the plane, the reader should consult \cite{GM}.

The situation in higher dimensions is more delicate. In 1974, Ziemer \cite{Z} found a topological sphere $\Omega\subset\RR^3$ whose boundary is 2-rectifiable with $\mathcal{H}^2(\partial\Omega)<\infty$, but whose harmonic measure is supported on a subset of zero area. This means that any analogue of Theorem \ref{FMsc} in space must impose extra non-topological conditions on $\partial\Omega$.
In this paper, we show the class of NTA domains (recalled in \S\ref{section4}) satisfy the forward direction of (\ref{FMsceq}).

\begin{theorem}\label{FMnta} Let $\Omega\subset\RR^n$ be NTA. If $\surf\res\partial\Omega$ is Radon (e.g. if $\surf(\partial\Omega)<\infty$), then $\partial\Omega$ is $(n-1)$-rectifiable and \begin{equation} \label{FMntaeq}
\omega(E)=0\Rightarrow \surf(E)=0\quad\text{for every Borel set } E\subset\partial\Omega.\end{equation}\end{theorem}

The proof of Theorem \ref{FMnta} that we present is based on the extension of Theorem B to $n \geq 3$ given by David and Jerison \cite{DJ}. Let  $\Omega\subset\RR^n$ be a NTA domain and assume its surface measure is \emph{Ahlfors regular}; that is, there exists a constant $C>0$ such that \begin{equation} \label{ahlfors} C^{-1}r^{n-1} \leq \surf(\Delta(Q,r))\leq Cr^{n-1}\quad\text{for all }Q\in\partial\Omega\text{ and }0<r<r_0.\end{equation} Using the existence of $(n-1)$-disks inside $B(Q,r)\cap\Omega$ and $B(Q,r)\setminus\Omega$ of radius $\geq c_0r$ (a weaker property than the corkscrew conditions enjoyed by NTA domains) and (\ref{ahlfors}), David and Jerison gave a geometric construction of Lipschitz domains $\Omega_L\subset B(Q,r)\cap \Omega$  such that $\surf(\partial\Omega_L\cap\partial\Omega)\geq c_1r^{n-1}$. In other words, there exists a Lipschitz approximation to $\Omega$ at each location and scale, which has substantial intersection in the boundary. Applying Dahlberg's theorem relating harmonic and surface measures on Lipschitz domains \cite{D} and a localization property for harmonic measure on NTA domains \cite{JK} yields Theorem B for NTA domains. (Theorem \ref{Lnta} was independently verified for 2-sided NTA domains by Semmes \cite{S} using a stopping time argument.)

\begin{oldthm}[David and Jerison, Semmes 1990] \label{Lnta}Let $\Omega\subset\RR^n$ be NTA. If (\ref{ahlfors}) holds, then $\omega\ll\sigma\ll\omega$ and $\omega\in A_\infty(\sigma)$.\end{oldthm}

The existence of big pieces of Lipschitz graphs implies that every NTA domain satisfying (\ref{ahlfors}) is uniformly rectifiable; this notion of quantitative rectifiability is developed in \cite{DS}. For a class of domains with non-doubling harmonic measure, on which a variant of the $A_\infty$ condition in Theorem C still holds, see Bennewitz and Lewis \cite{BL}. In the present work, we revisit David and Jerison's construction of Lipschitz  approximations to $\Omega$, focusing on the case when $\Omega$ is a corkscrew domain (e.g.\ when $\Omega$ is NTA). We make two observations. First surface measure on any corkscrew domain is automatically lower Ahlfors regular (Lemma \ref{lowersurf}). Second constructing a Lipschitz approximation at a given location and scale does not require surface measure be upper Ahlfors regular. Therefore one may relax the assumption that (\ref{ahlfors}) holds uniformly at all scales in Theorem \ref{Lnta}. This is our main result.

\begin{theorem}\label{macthm} Let $\Omega\subset\RR^n$ be NTA. Then the set \begin{equation}\label{macthmeq}A=\left\{Q\in\partial\Omega:\liminf_{r\downarrow 0}\frac{\surf(\Delta(Q,r))}{r^{n-1}}<\infty\right\}\end{equation} is $(n-1)$-rectifiable and $\omega\res A\ll\sigma\res A\ll\omega\res A$.\end{theorem}

If $\surf\res \partial\Omega$ is Radon, then $\surf(\partial\Omega\setminus A)=0$. Hence Theorem \ref{FMnta} follows directly from Theorem \ref{macthm}. It remains unknown whether the F. and M. Riesz theorem has a full analogue on NTA domains in higher dimensions. However, in view of Theorem \ref{macthm}, one can reverse the arrow in (\ref{FMntaeq}) if and only if

\begin{conjecture} \label{FMconj} Let $\Omega\subset\RR^n$ be NTA. If $\surf\res\partial\Omega$ is Radon, then \begin{equation} B= \left\{Q\in\partial\Omega: \lim_{r\downarrow 0} \frac{\surf(\Delta(Q,r))}{r^{n-1}}=\infty\right\}\end{equation} has harmonic measure zero. \end{conjecture}

The paper is organized as follows. In sections \ref{section2}--\ref{section3}, we demonstrate how to build a Lipschitz domain $\Omega_L$ inside of a corkscrew domain $\Omega$ at a given location $Q\in\partial\Omega$ and scale $r>0$ such that $\surf(\Delta(Q,r))<\infty$. At each step of the construction, we keep careful track of dependencies on parameters. The common boundary $\partial\Omega\cap\partial\Omega_L$ of a domain and its approximation has size determined by the dimension $n$ and corkscrew constant $M$ of $\Omega$; the Lipschitz constant and character of $\Omega_L$ only depends on $n$, $M$ and the ratio $\gamma=\surf(\Delta(Q,r))/r^{n-1}$. Section \ref{section21} outlines the construction of $\Omega_L$ using cones of a fixed aperture and reduces the approximation theorem (Theorem \ref{approxthm}) to choosing the correct slope of the defining cones (Proposition \ref{PropG}). In section \ref{section22} we quantify the relationship between harmonic and surface measures on the Lipschitz domain $\Omega_L$. The main tool is Jerison and Kenig's proof \cite{JKbull} of Dahlberg's theorem for star-shaped Lipschitz domains. The construction of $\Omega_L$ is completed in section \ref{section3}, where we verify Proposition \ref{PropG} by following the proof of the proposition in \cite{DJ}.

Section 4 is devoted to absolute continuity of harmonic measure on NTA domains of locally finite perimeter. We derive Theorem 1.2 from three main ingredients: good Lipschitz approximations to corkscrew domains (Theorem \ref{approxthm} and Lemma \ref{etalemma}), a localization property of harmonic measure on NTA domains (Lemma \ref{localize}), and a Vitali type covering theorem for Radon measures in $\RR^n$ (Theorem \ref{vitali}). An NTA domain is a corkscrew domain that also satisfies a Harnack chain condition. The proof of absolute continuity that we give actually shows Theorem 1.2 is valid on any corkscrew domain whose harmonic measure satisfies the conclusion of Lemma \ref{localize}.

Two applications of the main theorem are presented in section 5. First we prove that every Wolff snowflake (studied in \cite{W} and \cite{LVV}) has infinite surface measure. Second we compute the (upper) Hausdorff dimension of harmonic measure on NTA domains of locally finite perimeter: if $\Omega\subset\RR^n$ is NTA and $\surf\res\partial\Omega$ is Radon, then $\Hdim\omega=n-1$. This section can be read independently of \S\S\ref{section2}--\ref{section4}.

\section{Lipschitz Approximation to Corkscrew Domains}
\label{section2}

A closed set $\Sigma\subset\RR^n$ has \emph{big pieces of Lipschitz graphs} (often abbreviated \emph{BPLG}) if (i) $\surf\res \Sigma$ is Ahlfors regular and (ii) there are constants $\varphi>0$, $h>0$ and $r_0>0$ such that for every $Q\in\Sigma$ and $0<r<r_0$ there exists (up to an isometry in $\RR^n$) a graph $\Gamma\subset\RR^n$ of an $h$-Lipschitz function such that $\surf(B(Q,r)\cap \Sigma\cap \Gamma)\geq \varphi r^{n-1}$. In \cite{DJ}, David and Jerison proved if $\Sigma\subset\RR^n$ has Ahlfors regular surface measure and the open set $\RR^n\setminus\Sigma$ satisfies a ``two disk" condition, then $\Sigma$ has BPLG. Reading their proof carefully reveals that the upper bound in the Ahlfors regularity condition (\ref{ahlfors}) is not used to build Lipschitz graphs $\Gamma$. We verify this claim over the next two sections, in the special case that $\Sigma=\partial\Omega$ and $\Omega\subset\RR^n$ is a corkscrew domain. (This is the case applicable for Theorem \ref{macthm} and using the corkscrew condition instead of the two disk condition shortens the proof of several lemmas in \S\ref{section3}.)

\begin{definition}An open set $\Omega\subset\RR^n$ satisfies the \emph{corkscrew condition} with constants $M>1$ and $R>0$ provided that for every $Q\in\partial\Omega$ and every $0<r<R$ there exists a \emph{non-tangential point} $A=A(Q,r)\in\Omega$ such that $|A-Q|<r$ and $\dist(A,\partial\Omega)> r/M$.\end{definition}

\begin{definition}An open set $\Omega\subset\RR^n$ is a \emph{corkscrew domain} if $\Omega$ is connected and both $\Omega$ and $\RR^n\setminus\overline{\Omega}$ satisfy the corkscrew condition with constants $M>1$ and $R>0$.
\end{definition}

When $\Omega\subset\RR^n$ is a corkscrew domain, we write $A^+(Q,r)$ for non-tangential points in the interior $\Omega^+=\Omega$ and write $A^-(Q,r)$ for non-tangential points in the exterior $\Omega^-=\RR^n\setminus\overline{\Omega}$ of $\Omega$. Notice the definition does not require the exterior of a corkscrew domain to be connected.

Let us start with a simple application of the interior and the exterior corkscrew conditions. Surface measure on a corkscrew domain is always lower Ahlfors regular.  Here we normalize Hausdorff measure so that $\surf(B_{n-1}(0,1))=\omega_{n-1}$.

\begin{lemma}\label{lowersurf} There exists a constant $\beta=\beta(n,M)>0$ such that for every corkscrew domain $\Omega\subset\RR^n$ with constants $M>1$ and $R>0$, \begin{equation}\surf(\Delta(Q,r)) \geq \beta r^{n-1}\quad\text{for all }Q\in\partial\Omega\text{ and } 0<r<R.\end{equation}\end{lemma}

\begin{proof} Write $t=M/(M+1)$ and choose non-tangential points $a^\pm=A^\pm(Q,tr)$ of $\Omega^\pm$. Then $B(a^\pm,tr/M)\subset\Omega^\pm\cap B(Q,r)$. Let $\pi$ denote orthogonal projection onto a plane $P$ (of codimension 1) orthogonal to the line segment connecting $a^+$ and $a^-$. Assign $D^\pm$ to be the $(n-1)$-dimensional disk of radius $tr/M$ inside of $B(a^\pm, tr/M)$ and parallel to $P$. Because $D^+$ and $D^-$ lie in different connected components of $\RR^n\setminus\partial\Omega$ and the ball $B(Q,r)$ is convex, any line segment from $D^+$ to $D^-$ must intersect $\Delta(Q,r)$. Hence, since $\pi(D^+)=\pi(D^-)$ is a disk of radius $tr/M=r/(M+1)$ (Fig. 1), \begin{equation}\surf(\Delta(Q,r)) \geq \surf(\pi(\Delta(Q,r))) \geq \surf(\pi(D^\pm))=\omega_{n-1}\left(\frac{r}{M+1}\right)^{n-1}.\end{equation} Thus $\beta=\omega_{n-1}/(M+1)^{n-1}$ suffices.\end{proof}

\begin{figure}[t]
\includegraphics[width=.45\textwidth]{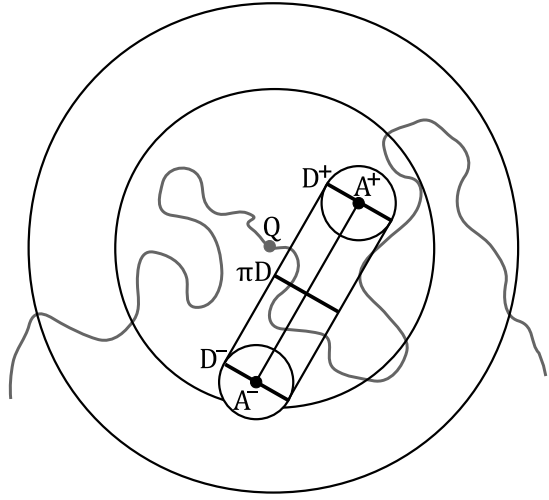}
\caption{Proof of Lemma 2.3}
\end{figure}

Our main goal in this section is to construct Lipschitz domains inside corkscrew domains with substantial intersection on the boundary. An important observation is the size of the big pieces of Lipschitz graphs depends only on the corkscrew constant.

\begin{theorem}\label{approxthm} Let $\Omega\subset\RR^n$ be a corkscrew domain with constants $M>1$ and $R>0$. There exists a constant $\psi=\psi(n,M)>0$ with the following property. For every point $Q\in\partial\Omega$ and every positive number $r<R$ such that $\surf(\Delta(Q,r))<\infty$, \begin{enumerate}
\item[($\star$)]
for every non-tangential point $a=A^+(Q,r/2)$ there exists a Lipschitz domain $\Omega_L\subset\RR^n$ such that $a\in\Omega_L\subset\Omega\cap B(Q,r)$, $\surf(\partial\Omega_L\cap\partial\Omega)\geq \psi r^{n-1}$, and $\partial\Omega_L\cap\partial\Omega$ is contained in (the rigid motion of) a single Lipschitz graph.
\end{enumerate} Moreover, we can find $\Omega_L$ in $(\star)$ so that the Lipschitz constant and character of $\Omega_L$ depend only on $n$, $M$ and the ratio $\gamma=\surf(\Delta(Q,r))/r^{n-1}$.
\end{theorem}

\begin{corollary}\label{corbplg} If $\Omega\subset\RR^n$ is a corkscrew domain and $\surf\res\partial\Omega$ is upper Ahlfors regular, then $\partial\Omega$ has big pieces of Lipschitz graphs.\end{corollary}

\begin{corollary}\label{corkrect} Let $\Omega\subset\RR^n$ be a corkscrew domain with constants $M>1$ and $R>0$. There exists a constant $h=h(n,M)>0$ such that \begin{equation}\label{corkrecteq}\{Q\in\partial\Omega: \surf(\Delta(Q,r))<\infty\text{ for some }r>0\}\end{equation} is contained, modulo a set of $\surf$-measure zero, in the countable union of sets $F_i(\RR^{n-1})$ where each function $F_i:\RR^{n-1}\rightarrow\RR^n$ has Lipschitz constant at most $h$.
\end{corollary}

\subsection{Constructing a Lipschitz Approximation}
\label{section21}

Let $\Omega\subset\RR^n$ be a corkscrew domain with constants $M>1$ and $R>0$, and fix $Q\in \partial\Omega$ and $0<r<R$ such that \begin{equation}\label{gammadefn}\surf(\Delta(Q,r))\leq\gamma r^{n-1}<\infty.\end{equation} We do not normalize the radius $r$, because we want to emphasize that the construction takes place at any fixed scale such that $\surf(\Delta(Q,r))/r^{n-1}<\infty$. Our immediate goal is to find a constant $\psi=\psi(n,M)>0$ such that $(\star)$ in Theorem \ref{approxthm} holds.

Let a non-tangential point $a=A^+(Q,r/2)$ of $\Omega$ be given. We select a piece of the boundary to approximate as follows. Pick any non-tangential point $b=A^-(Q,r/2)$ of $\Omega^-$. Then the line segment from $a$ to $b$ intersects $\Delta(Q,r/2)$ in some point $Q'$. After a harmless translation and rotation, we may assume that $Q'=0$, $a=(0,a_n)$ and $b=(0,-b_n)$ where $r/2M\leq a_n,b_n\leq r$.  Note that $B(a,r/2M)\subset B(Q,r)\cap\Omega$ and $B(b,r/2M)\subset B(Q,r)\cap \Omega^-$. Let $I_0=[-s/2,s/2]^{n-1}\subset\RR^{n-1}$ be the $(n-1)$-dimensional cube with side length \begin{equation}s=\frac{r}{2M\sqrt{n-1}}\end{equation} centered at the origin. Then $I_1=I_0\times\{a_n\}\subset B(a,r/4M)$. Write $\pi:\RR^{n}\rightarrow\RR^{n-1}$ and $f:\RR^{n}\rightarrow\RR$ for the orthogonal projections onto the first $(n-1)$ coordinates and the last coordinate of $\RR^{n}$, respectively. Fix a cone $\mathcal{C}$ opening upwards, \begin{equation}\label{conedefn}\mathcal{C}=\{z\in\RR^n:f(z)\geq h|\pi(z)|\},\end{equation} with parameter $h\gg 1$ to be chosen later. Let $\mathcal{T}$ denote the trapezoidal region \begin{equation}\mathcal{T}=\{y\in\RR^n:-b_n\leq f(y)\leq a_n\text{ and } (y+\mathcal{C})\cap f^{-1}(a_n)\subset I_1\}\end{equation} and write $T=\partial\Omega\cap\mathcal{T}$ for the portion of boundary in $\mathcal{T}$. We shall approximate $T$.

The reader can check that $I_0\times[a_n-s/4,a_n]\subset B(a,r/2M)\subset\Omega$ and similarly that $I_0\times[-b_n,-b_n+s/4]\subset B(b,r/2M)\subset\Omega^-$. Thus we know \begin{equation}\label{I1Tgap} -b_n+s/4\leq f(y)\leq a_n-s/4\quad\text{ for all }y\in T.\end{equation} If the slope $h$ of the cone $\mathcal{C}$ is sufficiently large, then the surface $T$ has large measure.

\begin{lemma}\label{piTLemma} If $h\geq h_0=16M\sqrt{n-1}$, then $\surf(\pi(T))\geq r^{n-1}/\left(4M\sqrt{n-1}\right)^{n-1}$.\end{lemma}

\begin{proof} If $h$ is sufficiently large, we claim that $\mathcal{T}$ contains the $(n-1)$-dimensional cube $I_2=[-s/4,s/4]^{n-1}\times\{-b_n\}$. Indeed first note $I_2\subset\mathcal{T}$ if and only if the corner $y=(s/4,\dots,s/4,-b_n)$ of the cube satisfies $(y+\mathcal{C})\cap f^{-1}(a_n)\subset I_1$. Hence if $z\in\RR^n$,  $f(z)=h|\pi(z)|$ and $-b_n+f(z)=f(y+z)=a_n$ (i.e.\ $y+z\in (y+\partial \mathcal{C})\cap f^{-1}(a_n)$) then $s/4+|\pi(z)|\leq s/2$ (i.e.\ $y+z\in I_1$) implies $I_2\subset\mathcal{T}$. That is, \begin{equation}\label{QtGoal} \frac{s}4+\frac{a_n+b_n}{h}\leq\frac{s}{2}\quad\Rightarrow\quad I_2\subset\mathcal{T}.\end{equation} Since $a_n+b_n\leq 2r=4Ms\sqrt{n-1}$, we find that $I_2\subset\mathcal{T}$ provided \begin{equation} \frac{s}{4}+\frac{4Ms\sqrt{n-1}}{h}\leq\frac{s}{2}.
\end{equation} Thus $I_2\subset\mathcal{T}$ when $h\geq h_0=16M\sqrt{n-1}$. Because every vertical line segment from $I_2\subset\Omega^-\cap\mathcal{T}$ to $I_1\subset\Omega^+\cap\mathcal{T}$ intersects $T$ and $\pi(I_2)\cap \pi(I_1)=\pi(I_2)$, we conclude \begin{equation} \surf(\pi(T))\geq \surf(\pi(I_2))= \left(\frac{s}{2}\right)^{n-1}= \left(\frac{r}{4M\sqrt{n-1}}\right)^{n-1}\end{equation} whenever $h\geq h_0$.
\end{proof}

\begin{figure}[t]
\includegraphics[width=.36\textwidth]{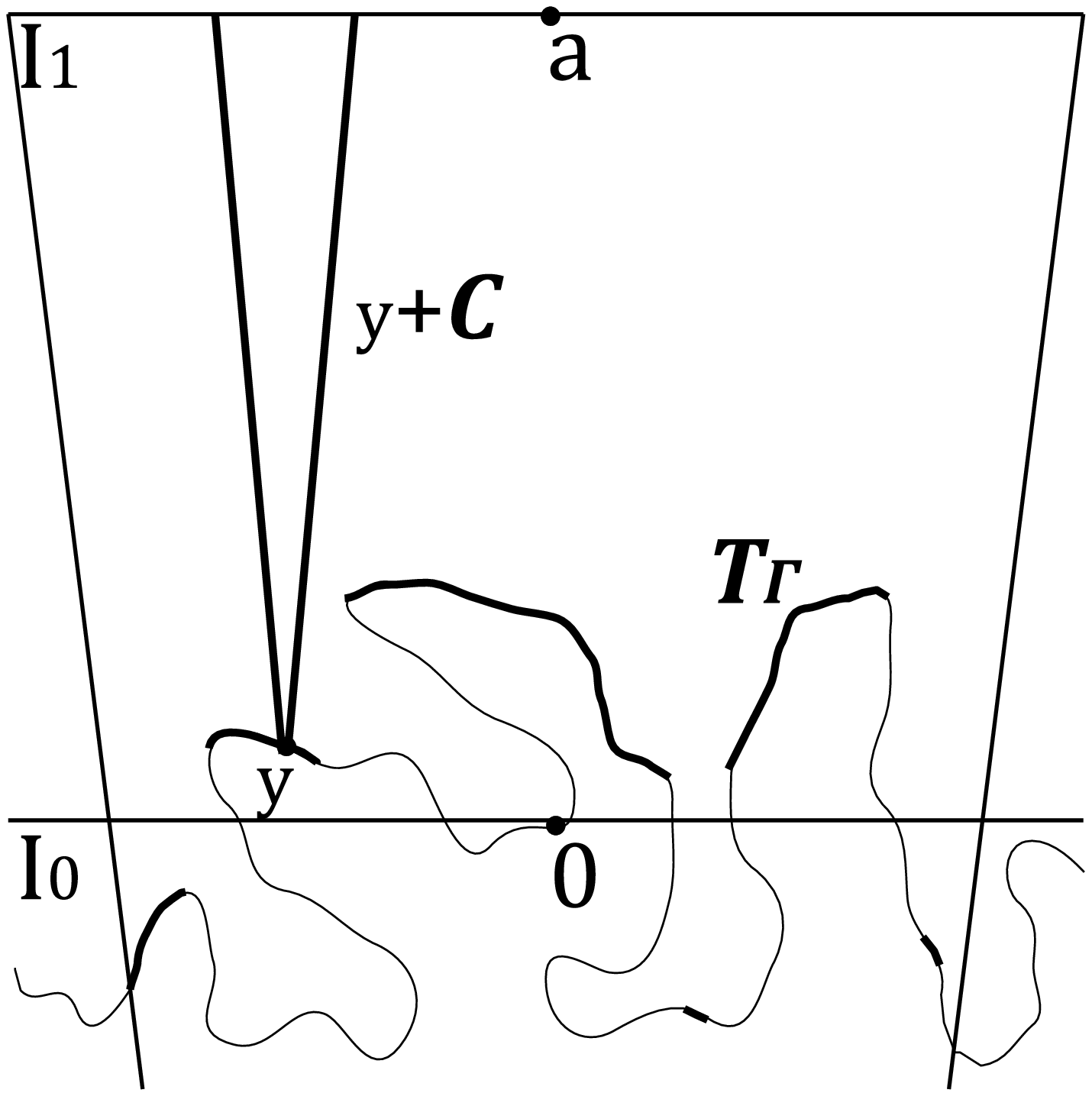}
\quad\quad
\includegraphics[width=.4\textwidth]{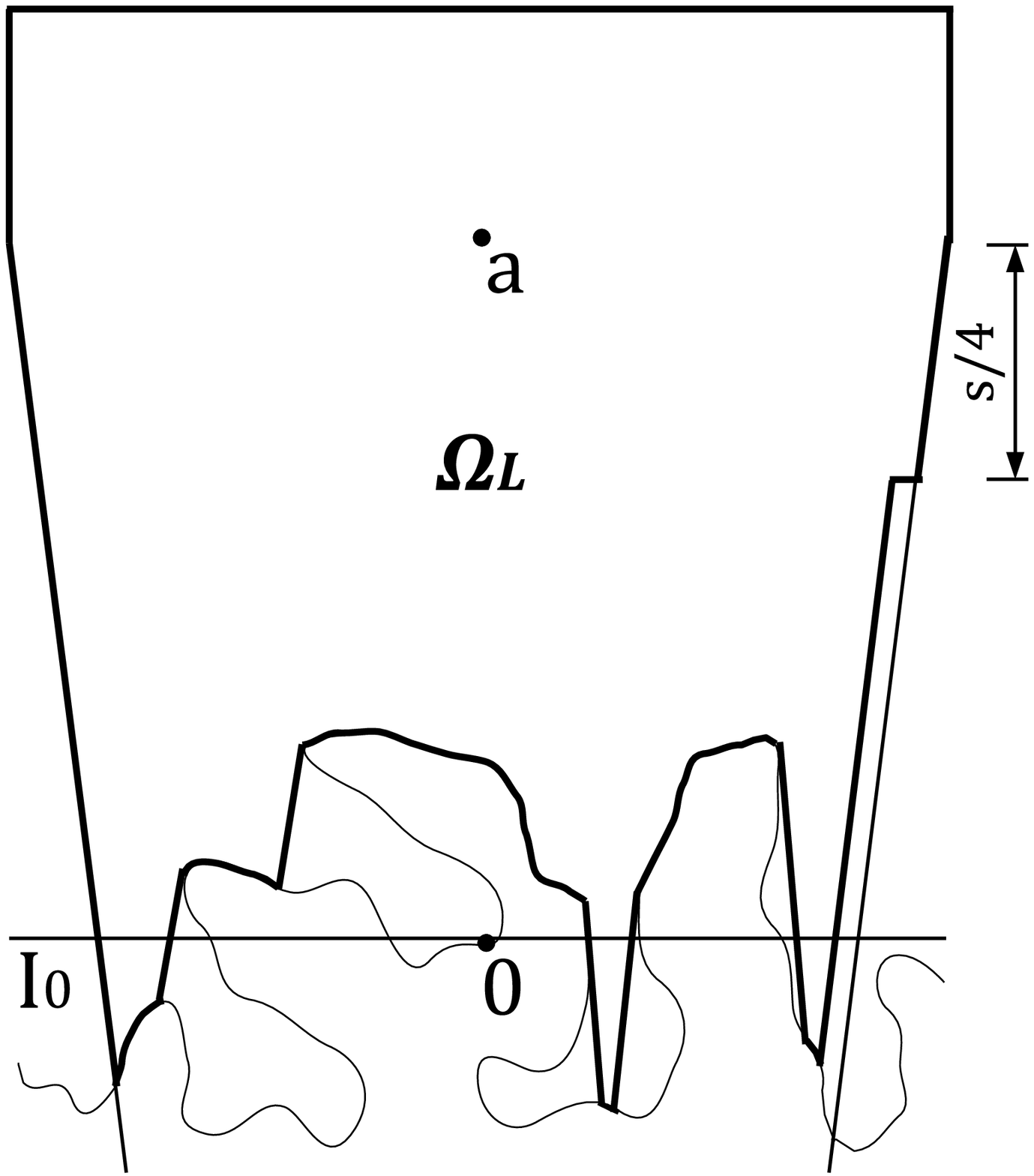}
\caption{What do $T_\Gamma$ and $\Omega_L$ look like?}
\end{figure}

We now use the cone $\mathcal{C}$ to identify a subset of $T$ which intersects a Lipschitz graph contained inside $\overline{\Omega\cap B(Q,r)}$ in a big piece. By a standard argument the set $T_\Gamma$ (Fig. 2), \begin{equation} T_\Gamma=\{y\in T:(y+\mathcal{C}) \cap T=\{y\}\}, \end{equation} sits inside the graph $\Gamma$ of a function $F:I_0\rightarrow\RR$ with Lipschitz constant at most $h$. Note by (\ref{I1Tgap}) replacing $F$ by $\max(\min(F,a_n-s/4),-b_n+s/4)$ does not effect $\Gamma\cap T$. Hence we may assume without loss of generality that $-b_n+s/4\leq F(x)\leq a_n-s/4$ for all $x\in I_0$. Define the domain \begin{equation}\Omega_L=\{(x,u)\in I_0\times \RR:u>F(x)\}\cap \interior\big(\mathcal{T}\cup (I_0\times [a_n,a_n+s/4])\big).\end{equation} That is, $\Omega_L$ is obtained by taking the area above $\Gamma$ inside $\mathcal{T}$ and then extending upwards so that $B(a,s/4)$ lies inside the domain. Since the vertical extension satisfies $I_0\times[a_n,a_n+s/4]\subset B(a,r/2M)$, $\Omega_L\subset \Omega\cap B(Q,r)$ and $\Omega_L$ is a Lipschitz domain with intersection $\partial\Omega_L\cap\partial\Omega=\Gamma\cap T=T_\Gamma$. Notice that $\Omega_L$ can be covered by $c(n)$ Lipschitz graphs with  constant at most $h$.

To select the slope $h$ of $\mathcal{C}$ large enough so that $\surf(T_\Gamma)\geq \psi r^{n-1}$ for some constant $\psi=\psi(n,M)>0$, we need the following claim, a slight modification of the geometric proposition in \cite{DJ}.

\begin{proposition}\label{PropG} For all $\varepsilon>0$, there exists $h\geq h_0$ depending only on  $n$, $M$, $\gamma$ and $\varepsilon$ such that $\surf(\pi(T)\setminus \pi(T_\Gamma))\leq \varepsilon r^{n-1}$.\end{proposition}

The proof of Proposition \ref{PropG} is a long but fairly straightforward application of the corkscrew condition on the exterior $\Omega^-$ of $\Omega$, the lower Ahlfors regularity of $\surf\res\partial\Omega$, and the upper bound $\surf(\Delta(Q,r))\leq \gamma r^{n-1}$; details are postponed until \S\ref{section3}. First let us finish studying the Lipschitz approximation $\Omega_L$ to $\Omega\cap B(Q,r)$.

\begin{proof}[Proof of Theorem 2.4] By Lemma \ref{piTLemma} and Proposition \ref{PropG}, \begin{equation}\begin{split}\surf(T_\Gamma) \geq \surf(\pi(T_\Gamma)) &= \surf(\pi(T)) - \surf(\pi(T)\setminus\pi(T_\Gamma))\\
&\geq \left(\frac{1}{\left(4M\sqrt{n-1}\right)^{n-1}}-\varepsilon\right) r^{n-1}.\end{split}\end{equation} Choosing $\varepsilon>0$ sufficiently small (equivalently choosing $h\geq h_0$ sufficiently large), we conclude that $\surf(\partial\Omega_L\cap\partial\Omega)=\surf(T_\Gamma)\geq \psi r^{n-1}$ with \begin{equation} \psi=\frac{1}{\left(8M\sqrt{n-1}\right)^{n-1}}.\end{equation} The constant $\psi$ only depends on $n$ and $M$; the Lipschitz constant of the graph $\Gamma$ and the Lipschitz character of the domain $\Omega_L$ are determined by $h$ and thus by $n$, $M$, $\gamma$.\end{proof}

\subsection{Harmonic Measure and the $A$-infinity Condition}
\label{section22}

Next we compare harmonic measure $\omega_L$ and surface measure $\sigma_L$ on $\Omega_L$ (Lemma \ref{etalemma}) using constants depending only on $n$, $M$ and $\gamma$.  A theorem of Dahlberg \cite{D} asserts a strong relationship between harmonic measure and surface measure exists on any bounded Lipschitz domain.

\begin{theorem}[\cite{D} Theorem 3] \label{Dthm}Let $D\subset\RR^n$ be a bounded Lipschitz domain equipped with harmonic measure $\omega_D$ and surface measure $\sigma_D$. Then $\omega_D\in A_\infty(\sigma_D)$.\end{theorem}

To use Theorem \ref{Dthm} effectively, we must understand the dependence of constants in the $A_\infty$ condition on the features of a Lipschitz domain. There are two proofs of Theorem \ref{Dthm} (Dahlberg \cite{D}, Jerison and Kenig \cite{JKbull}) and each proof first establishes the theorem on a special class of star-shaped Lipschitz domains. (A domain $D$ is called \emph{star-shaped} with \emph{center} $c$ if every open line segment from $c$ to $\partial D$ lies inside $D$.) Thus the $A_\infty$ condition in Theorem \ref{Dthm} may depend on how star-shaped Lipschitz domains cover the original domain. To clarify this dependence, we need to introduce some notation.
Let $D$ be a star-shaped Lipschitz domain with center $c$, and write $\vec n_y$ for the outer unit normal to $\partial D$ defined at $\surf$-a.e.\ $y\in\partial D$. Define the \emph{angle function} $\vartheta_{D,c}(y)$ (see Fig. 3) by \begin{equation}\cos \vartheta = \frac{\langle \vec n_y,y-c\rangle}{\|y-c\|},\quad 0\leq \vartheta\leq \pi.\end{equation} Note $\vartheta_{D,c}(y)$ is defined at $\surf$-a.e.\ $y\in\partial D$. Since $D$ is star-shaped with center $c$, the angle function $\vartheta_{D,c}(y)\in[0,\pi/2]$ almost surely.

\begin{figure}[t]
\includegraphics[width=.35\textwidth]{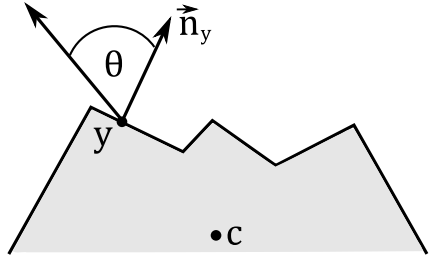}
\caption{The angle function $\vartheta_{D,c}(y)$}
\end{figure}

The following proposition is adapted from the proof of Theorem \ref{Dthm} in \cite{JKbull}.

\begin{proposition}\label{angles} Let $D\subset\RR^n$ be a bounded star-shaped Lipschitz domain with center $c\in D$. Let $h$ be the Lipschitz constant of $D$ and assume there exists radii $\rho_1,\rho_2>0$ such that $B(c,\rho_1)\subset D\subset B(c,\rho_2)$. For all $\vartheta_0<\pi/2$, there exists a constant $C=C(n,h,\rho_2/\rho_1,\vartheta_0)$ with the following property. If $\vartheta_{D,c} \leq \vartheta_0$ a.e. on $\Delta_D(y_0,r_0)$ for some $y_0\in\partial D$ and $r_0>0$, then the Radon-Nikodym derivative $k=d\omega^c_D/d\sigma_D$ of harmonic measure with pole at $c$ with respect to surface measure satisfies the reverse H\"older inequality \begin{equation} \label{2holder} \left(\dashint_{\Delta(y,r)}k^2d\sigma_D\right)^{1/2} \leq C \dashint_{\Delta(y,r)} kd\sigma_D\quad\text{for every }\Delta_D(y,r)\subset\Delta_D(y_0,r_0).\end{equation}
\end{proposition}

Here the dashed integral $\dashint_{\Delta} k d\sigma=(\sigma(\Delta))^{-1}\int_\Delta k d\sigma$ denotes an average. By the theory of $A_\infty$ weights, if condition (\ref{2holder}) holds, then (\ref{Lsceq}) also holds on $\Delta_D(y_0,r_0)$ with constants $\delta$ and $\varepsilon$ which depend only on $n$ and $C$.

Let us return our attention to the comparison of harmonic measure and surface measure on $\Omega_L$.
First we cover $\Omega_L$ using two types of star-shaped Lipschitz domains (see Fig. 4) and estimate harmonic measure in each case separately.

\begin{figure}[t]
\includegraphics[width=.4\textwidth]{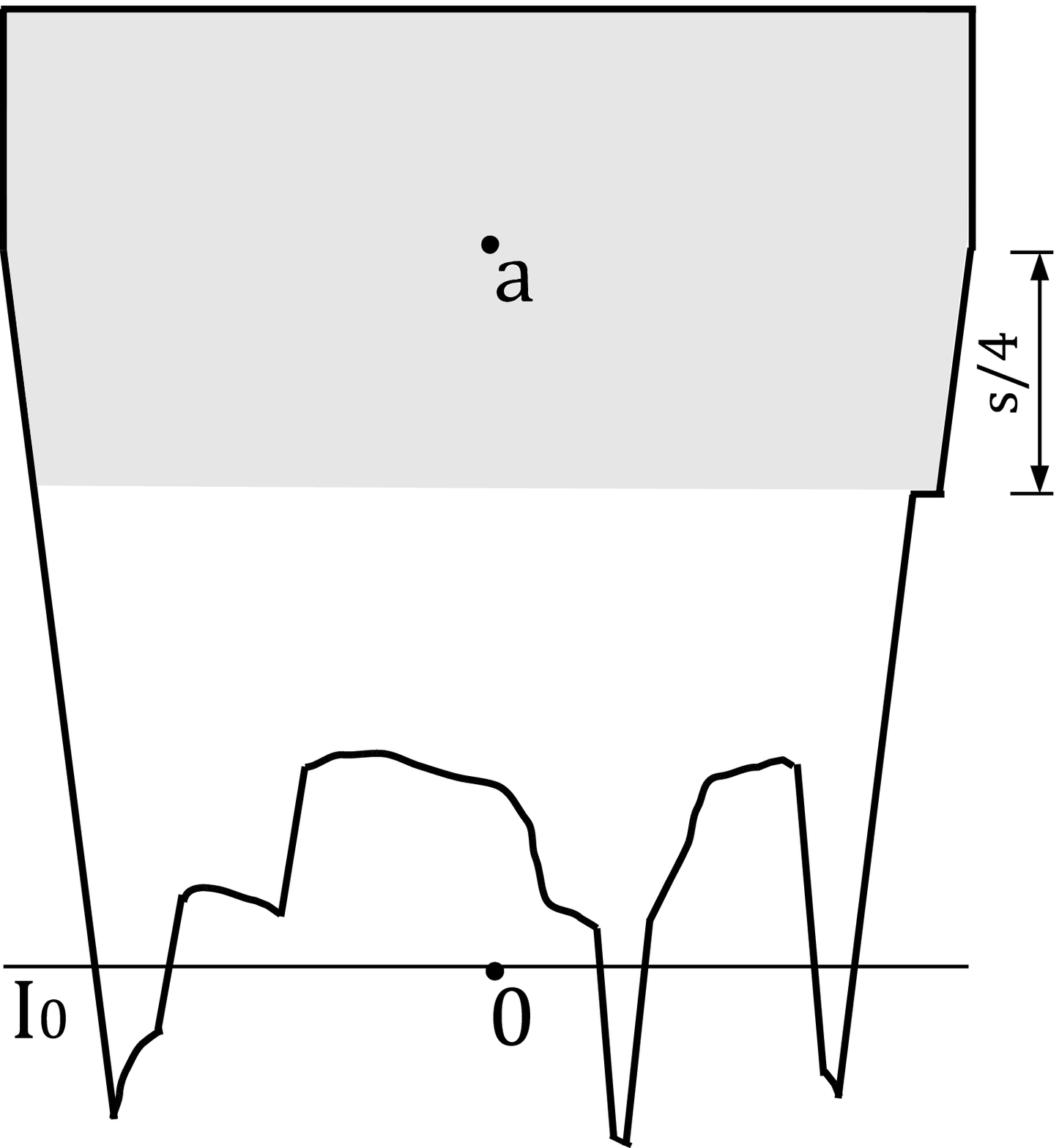}
\includegraphics[width=.4\textwidth]{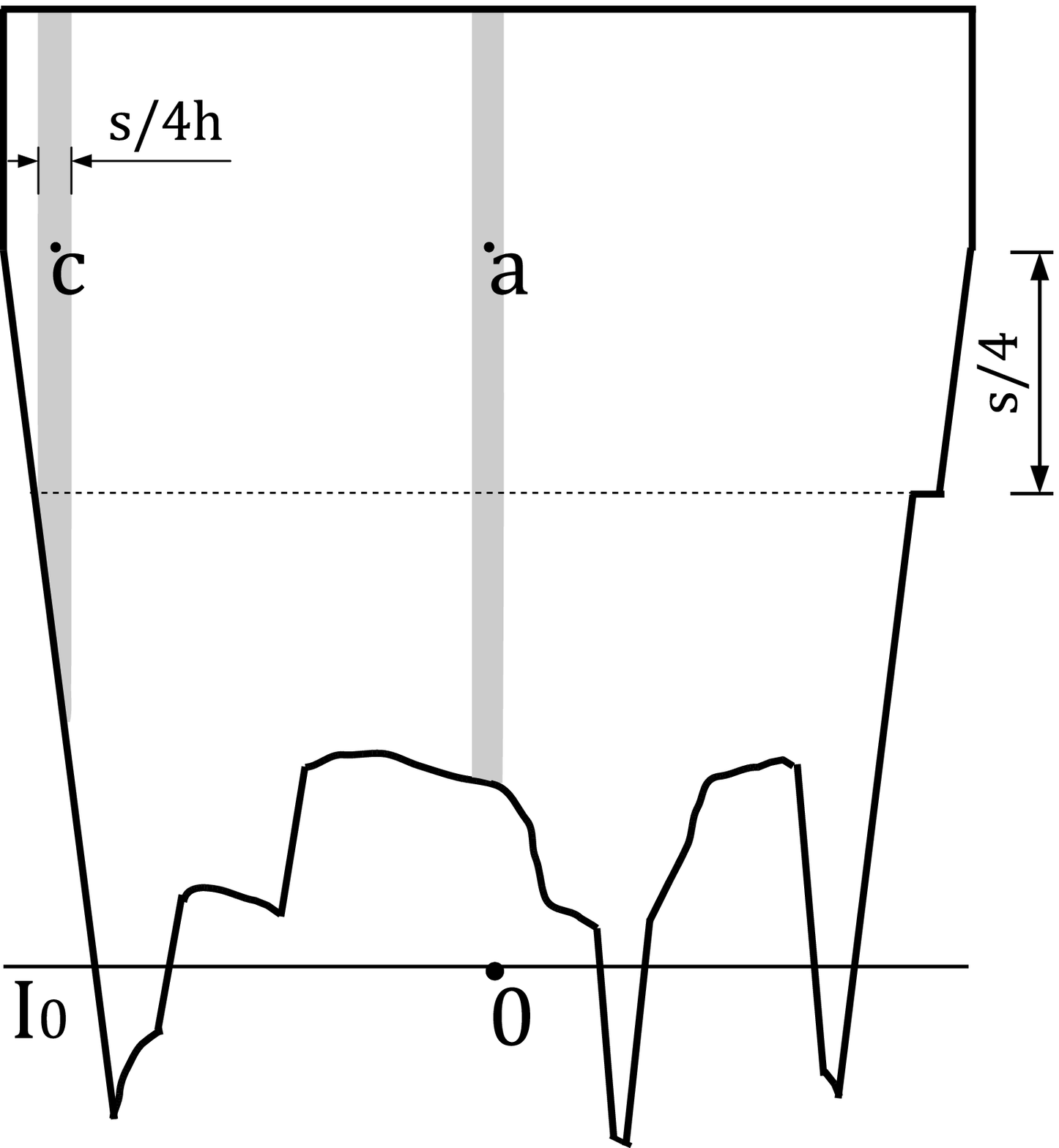}
\caption{Star-shaped domains in Lemmas 2.11 and 2.12}
\end{figure}

\begin{lemma}\label{toplemma} Up to a dilation, the Lipschitz domain \begin{equation}\Dtop = \{y\in \Omega_L: f(y)>a_n-s/4\}\end{equation} is determined by $n$ and $h$. Moreover, there exists a constant $\eta_0=\eta_0(n,M,\gamma)$ such that
\begin{equation}
\label{topinfty}\omega^a_L(E)\leq \eta_0\Rightarrow \surf(E)\leq\frac{\psi}{4}r^{n-1}\quad\text{for all }E\subset \partial\Omega_L\cap\partial \Dtop\end{equation} Here $\omega_L^a$ denotes harmonic measure of $\Omega_L$ with pole at $a$.
\end{lemma}

\begin{proof} The Lipschitz function $F:I_0\rightarrow\RR$ used to build $\Omega_L$ satisfied $F(x)\leq a_n-s/4$ for all $x\in I_0$. Hence $\Dtop$ is a fixed subset of the region $\interior(\mathcal{T}\cup (I_0\times[a_n,a_n+s/4]))$ and determined by $n$ and $h$. By Theorem \ref{Dthm}, $\omega^a_{\mathrm{top}}\in A_\infty(\sigma_{\mathrm{top}})$, where $\omega_{\mathrm{top}}^a$ is harmonic measure on $\Dtop$ with pole at $a$ and $\sigma_{\mathrm{top}}=\surf\res \partial\Dtop$. Thus there exist constants $p,q>0$ depending only on $n$ and $h$ such that \begin{equation}\surf(E)\leq  p\left[\omega^a_{\mathrm{top}}(E)\right]^{q} \surf(\Dtop) \quad\text{for all } E\subset\partial\Dtop.\end{equation} On one hand, $\omega^a_{\mathrm{top}}(E)\leq \omega^a_L(E)$ for all $E\subset \partial \Omega_L\cap\partial\Dtop$, by the maximum principle. On the other hand, $\surf(\Dtop)\leq C_0 r^{n-1}$ for some $C_0=C_0(n,h)$. Therefore, \begin{equation}\surf(E)\leq C_0p \left[\omega^a_L(E)\right]^q r^{n-1}\quad\text{for all }E\subset\partial\Omega_L\cap \partial\Dtop,\end{equation} and the constant $\eta_0 = \left(\psi/4C_0p\right)^{1/q}$ depending only on $n$, $M$ and $\gamma$ suffices.
\end{proof}

\begin{lemma}\label{starlemma} Set $h^*=h\sqrt{n-1}$. For each $c\in I_1$ define \begin{equation} D_c=\{y\in\Omega_L:\|\pi(y)-\pi(c)\|_{\infty}<s/8h^*\}\end{equation} where $\|x\|_{\infty}=\max_i|x_i|$ for all $x\in\RR^{n-1}$. Assume  $\|\pi(c)\|_\infty \leq s/2-s/4h^*-s/8h^*$. Then $D_c$ is a star-shaped Lipschitz domain with center $c$, the Lipschitz constant of $D_c$ is at most $h$ and $B(c,s/8h^*)\subset D_c\subset B(c,4Ms\sqrt{n-1})$. Moreover, there exists a constant $\vartheta_1=\vartheta_1(n,M,\gamma)<\pi/2$ such that $\vartheta_{D,c}\leq \vartheta_1$ on $\partial D_c$ almost surely.
\end{lemma}
\begin{proof} Fix any point $c\in I_1$ such that $\|\pi(c)\|_{\infty}\leq s/2-s/4h^*-s/8h^*$. This condition on $c$ guarantees that  $\|\pi(y)\|_\infty \leq s/2-s/4h^*$ for all $y\in D_c$ and the top portion of $D_c$ is a box that (up to translation) is independent of $c$:  \begin{equation}\label{topisbox}D_c\cap \Dtop=(\pi(c)-s/8h^*,\pi(c)+s/8h^*)^{n-1}\times (a_n-s/4,a_n+s/4).\end{equation} Hence $B(c,s/8h^*)\subset D_c$. The inclusion $D_c\subset B(c,4Ms\sqrt{n-1})$ follows from (\ref{I1Tgap}). The bottom portion of $D_c$ (i.e.\ $D_c\cap f^{-1}(-\infty, a_n-s/4)$) is the area above the graph \begin{equation} \Gamma_c= \{y\in\partial\Omega_L\cap\Gamma:\|\pi(y)-\pi(c)\|_\infty<s/8h^*\}.\end{equation}
For all $p>0$ let $\mathcal{C}_p=\{z\in\RR^n:f(z)\geq p|\pi(z)|\}$ denote the cone opening upwards with slope $p$. The cone used to define $\Omega_L$ above was $\mathcal{C}=\mathcal{C}_{h}$.
If $y\in \Gamma_c$, then $c\in y+\mathcal{C}_{2h}$ since $f(c)-f(y)=a_n-f(y)\geq s/4$ and $|\pi(c)-\pi(y)|\leq s/8h$. Because $\mathcal{C}_{2h}\subset\mathcal{C}_h$ and  $(y+\mathcal{C}_h)\cap\Gamma=\{y\}$, the open line segment from $c$ to $y\in \Gamma_c$ is contained in $D_c$. Thus $D_c$ is star-shaped with respect to $c$. It remains to bound the angle function.

Suppose $\vec n_y$ is an outer normal to $\partial D_c$ defined at $y\in\Gamma_c$. On one hand, $\vec n_y \in -{C}_{1/h}$, since $(y+\mathcal{C}_h)\cap\Gamma=\{y\}$. On the other hand, $y-c\in -\mathcal{C}_{2h}$. The greatest angle between a vector $\vec v\in -\mathcal{C}_{1/h}$ and a vector $\vec w\in -\mathcal{C}_{2h}$ is obtained by $\vec v=(1,0,\dots,-1/h)$ and $\vec w=(-1,0,\dots,-2h)$; in this case, \begin{equation}\cos \vartheta = \frac{\langle \vec v,\vec w\rangle}{\|\vec v\|\|\vec w\|} =\frac{1}{(1+h^{-2})^{1/2}(1+4h^2)^{1/2}}\geq  \frac{1}{h\sqrt{10}}.\end{equation} We conclude $\vartheta_{D,c}(y) \leq \cos^{-1}(1/h\sqrt{10})<\pi/2$ for almost every $y\in\Gamma_c$. Bounding $\vartheta_{D,c}$ on $\partial D_c\setminus \Gamma_c$ (that is, on sides of a box) is easier and left to the reader.\end{proof}

Equipped with Proposition \ref{angles}, Lemma \ref{toplemma} and Lemma \ref{starlemma}, we are ready to compare harmonic measure and surface measure on $\Omega_L$.

\begin{lemma}\label{etalemma} There exists a constant $0<\eta<1$ depending only on $n$, $M$ and $\gamma$ with the following property. For every Borel set $E\subset\partial\Omega_L$, \begin{equation}
\label{etainfty}\omega^a_L(E)\leq \eta\Rightarrow \surf(E)\leq\frac{\psi}{2}r^{n-1}.\end{equation} Here $\omega_L^a$ denotes harmonic measure on $\Omega_L$ with pole at $a$.
\end{lemma}

\begin{proof}
Choose points $c_1,\dots, c_{i_0}\in I_1$ such that $\|\pi(c_i)\|_\infty \leq s/2-s/4h^*-s/8h^*$ and \begin{equation}(-s/2+s/4h^*,s/2-s/4h^*)^{n-1}\times\{a_n\}\subset \bigcup_i D_{c_i}.\end{equation} We can make this choice so that $i_0$ only depends on $n$ and $h$. Notice that \begin{equation}\label{etacover}\partial\Omega_L \subset \partial \Dtop \cup\bigcup_i \Gamma_{c_i}.\end{equation}
The points $c_i$ and  $a$ lie inside $\Dtop$, at a uniform distance away from $\partial\Dtop$. By Lemma \ref{toplemma} and  Harnack's inequality, there exists a constant $C_1=C_1(n,h)>1$ such that \begin{equation}\label{L2harnack}\omega^{c_i}_L(E)\leq C_1 \omega^a(E)\quad\text{for all }1\leq i\leq i_0\text{ and } E\subset\partial\Omega_L.\end{equation} Thus, in view of (\ref{topinfty}), (\ref{etacover}) and (\ref{L2harnack}), to prove Lemma \ref{etalemma} it suffices to display $\eta=\eta(n,M,\gamma)\in (0,\eta_0)$ small enough so that \begin{equation}\label{etagoal} \omega_L^{c_i}(E \cap \Gamma_{c_i})\leq C_1 \eta \Rightarrow \surf(E\cap \Gamma_{c_i}) \leq \frac{\psi}{4 i_0} r^{n-1}\end{equation} for all $1\leq i\leq i_0$ and $E\subset \partial\Omega$.

By Proposition \ref{angles} and Lemma \ref{starlemma}, $D_{c_i}$ is a star-shaped Lipschitz domain whose harmonic measure $\omega_D^{c_i}$ satisfies $(\ref{2holder})$ on every disk  for some constant $C_2$ depending only on $n$, $M$ and $\gamma$. An equivalent form of the $A_\infty$ condition states that for every $\varepsilon>0$ there exists $\delta=\delta(n,C_2,\varepsilon)>0$ such that \begin{equation} \omega_{D}^{c_i}(E\cap \Gamma_{c_i})\leq \delta \Rightarrow \surf(E\cap\Gamma_{c_i})\leq\varepsilon \surf(\partial D_{c_i}).\end{equation} But $\surf(\partial D_{c_i})\leq C_3(n,h) r^{n-1}$, so we can assign $\varepsilon=C_3^{-1}\psi/4i_0$ to find a constant $\delta=\delta(n,M,\gamma)>0$ such that \begin{equation}\label{etaeq1}\omega_{D}^{c_i}(E\cap \Gamma_{c_i})\leq \delta \Rightarrow \surf(E\cap\Gamma_{c_i})\leq \frac{\psi}{4i_0} r^{n-1}.\end{equation} Set $\eta = \min(\eta_0,\delta/C_1)$ so that $\eta$ only depends on $n$, $M$ and $\gamma$. Then (\ref{etagoal}) follows from (\ref{etaeq1}) and the maximum principle.
\end{proof}

\section{Proof of Proposition 2.8}
\label{section3}

We continue to assume the notation adopted in \S\ref{section21}. Let $\varepsilon>0$ be given. Our goal is to choose the slope $h\geq h_0$ of the cone $\mathcal{C}$ so that $\surf(\pi(T)\setminus\pi(T_\Gamma))\leq \varepsilon r^{n-1}$. In the course of exposition we shall introduce several constants and indicate their dependence on previously defined quantities; each one  will ultimately depend on at most $n$ (dimension), $M$ (corkscrew constant), $\gamma$ (upper bound at scale $r$) and $\varepsilon$. Following \cite{DJ}, we start by breaking up the set $\pi(T)\setminus\pi(T_\Gamma)$ into manageable pieces.

\begin{lemma}\label{Helper1} Let $H:\RR^{n-1}\rightarrow[0,\infty]$ be the function\begin{equation}
H(x)=\sup\left\{\frac{\surf(\Delta(Q,r)\cap \pi^{-1}(I))}{\surf(I)}
:I\subset\RR^{n-1}
\text{ is a cube and }x\in I\right\}.
\end{equation} If $\lambda_H(N)=\{x\in \RR^{n-1}:H(x)\geq N\}$, then $\surf(\lambda_H(N))\leq 5^{n-1}\gamma r^{n-1}/N$.\end{lemma}
\begin{proof} Just note $H(x)$ is the maximal function of the the measure \begin{equation}\mu=\pi_\sharp(\surf\res\Delta(Q,r))\end{equation} with respect to the Lebesgue measure on $\RR^{n-1}$. By the Hardy-Littlewood maximal theorem (for example, see Theorem 2.19 in \cite{M}), $\surf(\lambda_H(N))\leq 5^{n-1}\mu(\RR^{n-1})/N$. By (\ref{gammadefn}), $\mu$ has total mass $\mu(\RR^{n-1})\leq \gamma r^{n-1}$.\end{proof}

Note that the upper bound $\surf(\Delta(Q,r))\leq\gamma r^{n-1}$ on surface measure was used in the proof of Lemma \ref{Helper1}. It will not be used again until after proof of Lemma \ref{Helper7}.

For the remainder of the section, fix $N=2\cdot 5^{n-1}\gamma/\varepsilon$ so that the set $\Lambda=\lambda_H(N)$ of points in $\RR^{n-1}$ where the maximal function $H(x)$ is large has small measure, $\surf(\Lambda)\leq (\varepsilon/2)r^{n-1}$. We must control the size of $\pi(T)\setminus\pi(T_\Gamma)$ in $\Lambda^c=\RR^{n-1}\setminus\Lambda$. For each $y\in\mathcal{T}$, let $L(y)=[y,(\pi(y),a_n)]$ denote the vertical segment above $y$ in $\mathcal{T}$. Then the set of points \begin{equation}T_E=\{y\in T:L(y) \cap T=\{y\}\}\end{equation} denotes the ``top edge" of $\partial\Omega$ inside $\mathcal{T}$. Observe that $T_\Gamma\subset T_E\subset T$ and $\pi(T)=\pi(T_E)$. Let $\alpha$ be a large power of 2 to be chosen later (after Lemma \ref{Helper7}) and abbreviate $s_p=s/\alpha^p$ for all $p$. For each integer $k\geq 0$, define the set $F_k\subset I_0$ (see Fig. 5), \begin{equation} F_k=\{\pi(y): \exists\  y,z\in T_E \text{ such that } z\in y+\mathcal{C} \text{ and } s_k\leq f(z)-f(y)\leq  s_{k-1}\}.\end{equation}

\begin{figure}[t]
\includegraphics[width=.4\textwidth]{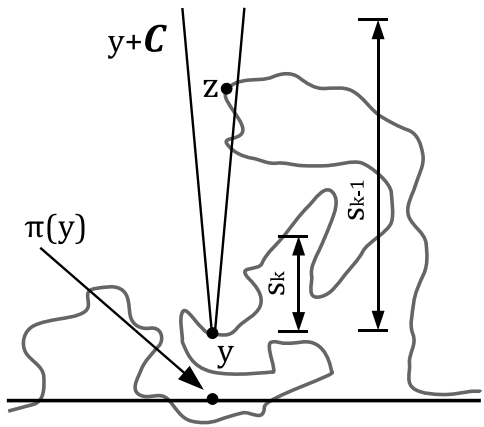}
\caption{A point $\pi(y)\in F_k$}
\end{figure}

By (\ref{I1Tgap}), if we choose $\alpha s=s_{-1}\geq a_n+b_n-s/2$, each bad point $x\in \pi(T)\setminus \pi(T_\Gamma)$ belongs to at least one $F_k$. For this reason, we will stipulate that \begin{equation}\alpha\geq 4M\sqrt{n-1} = 2r/s > s^{-1}(a_n+b_n-s/2).\end{equation} Thus, we can prove Proposition \ref{PropG} by verifying $\surf\left(\bigcup_k F_k\right)\leq \varepsilon r^{n-1}$, or because of Lemma \ref{Helper1}, by demonstrating that \begin{equation}\label{goal3} \sum_k\surf(F_k\setminus \Lambda)\leq\frac{\varepsilon}{2} r^{n-1}\quad\text{for some }h\geq h_0.\end{equation}

\begin{remark} We do not assert that $T_E$ or $F_k$ are measurable. While this fact is irksome, it does not hinder the proof. A careful reader will observe that we only use countable subadditivity of the outer measure $\surf$ in coverings involving $F_k$.
\end{remark}

The next lemma captures a simple idea. If the surface $\partial\Omega\cap\pi^{-1}(I)$ over a cube $I\subset\RR^{n-1}$ has a big vertical span relative to the width of $I$, then the maximal function is big on $I$ (by lower Ahlfors regularity). Thus the maximal theorem limits the frequency of ``vertical jumps" in $\partial\Omega$. Let $\beta=\beta(n,M)$ be the constant given in Lemma \ref{lowersurf}. Define the $n$-dimensional box $\mathcal{R}=2I_0\times [-b_n,a_n]$ and note $\mathcal{T}\subset\mathcal{R}\subset B(Q,r)$.

\begin{lemma} \label{Helper2} Let $I\subset 2I_0$ be a cube of side length $t$. Suppose one can find line segments $L$ in $\Omega\cap \mathcal{R}$ and $L'$ in $\Omega^-\cap\mathcal{R}$ such that $\pi(L)$ and $\pi(L')$ belong to $I$ and $f(L)\cap f(L')$ is a segment of length $\geq \lceil 4^{n-1}\beta^{-1}N\rceil t$. Then the cube $I$ belongs to $\Lambda$.\end{lemma}

\begin{proof} Select $\lceil 4^{n-1}\beta^{-1}N\rceil$ points $u_i$ in $f(L)\cap f(L')$ such that $|u_i-u_j|\geq t$ for $i\neq j$. Then for each $u_i$ the horizontal line segment in $f^{-1}(u_i)$ which joins $L$ to $L'$ intersects $\partial\Omega$ at some point $y_i$, because $L$ and $L'$ belong to different components of $\mathcal{R}\setminus\partial\Omega$. The balls $\Delta(y_i,t/2)$ are disjoint sets and by Lemma \ref{lowersurf} (note $t\ll 2s<r<R$), \begin{equation}\surf\left(\bigcup_i \Delta(y_i,t/2)\right)\geq (4^{n-1}\beta^{-1} N)\beta \left(\frac{t}{2}\right)^{n-1}= N\surf(2I).\end{equation} Hence, $2I\subset\Lambda$, since $\Delta(y_i,t/2)\subset \Delta(Q,r)\cap\pi^{-1}(2I)$ for each $y_i$ (one easily checks that the distance of $\mathcal{R}$ to $\RR^n\setminus B(Q,r)$ is farther than $t/2$). In particular, $I\subset\Lambda$.
\end{proof}

It will be convenient to work with a ``dyadic" decomposition of $I_0$. We say that a cube $I\subset I_0$ is \emph{admissible} if the cube $\{(1/2,\dots,1/2)+x/s:x\in I\}$ is dyadic in the usual sense. Hence $I_0$ is admissible and every admissible cube $I\subset I_0$ of side length $s/2^i$ is the almost disjoint union of $2^{j(n-1)}$ admissible cubes of side length $s/2^{i+j}$. Moreover, given any cube $I\subset I_0$ of side length $t$, we can find an admissible cube contained in $I$ of side length $\geq t/4$. The following ``search lemma" locates admissible cubes inside $\Lambda\cup F_k^c$.

\begin{lemma}\label{Helper3} There exist constants $C_1=C_1(n,N,\alpha,\beta)$ and $C_2=C_2(n,M,N,\beta)$ with the following property. If $I\subset I_0$ is any cube with side length $t$ satisfying \begin{equation}
\frac{C_1}{h} s_{k} \leq t \leq s_k,\end{equation} then one can find an admissible cube $J\subset I\cap (\Lambda\cup F_k^c)$ of side length $\geq t/C_2$.
\end{lemma}

\begin{proof} Let $I\subset I_0$ with side length $t$, $C_1 h^{-1} s_{k}\leq t\leq s_k$ be given. Since any cube contains an admissible cube of comparable size, we may first search for a cube $J\subset I \cap(\Lambda \cup F_k^c)$  which is not necessarily admissible.

For every cube $\mathcal{Q}\subset\RR^{n-1}$ such that $\pi(T)\cap \mathcal{Q}\neq \emptyset$, define \begin{equation} \label{searcheq1} \lambda(\mathcal{Q})=\max\{f(y):y\in T_E\text{ and }\pi(y)\in \mathcal{Q}\}.\end{equation} The maximum in (\ref{searcheq1}) is realized, because $T_E$ is the ``top edge" of $T$ and $T$ is compact. Suppose there exists a cube $\mathcal{Q}\subset I$ of side length $qt$ such that $\lambda(\mathcal{Q})\leq \lambda\left(\frac12\mathcal{Q}\right)+s_k/8$. Pick $w\in T_E$ such that $\pi(w)\in\frac12 \mathcal{Q}$ and $f(w)=\lambda(\frac12\mathcal{Q})$, and let  $c=A^-(w,qt/8)$ be a non-tangential point of $\Omega^-$. Then $B(c,qt/8M)\subset \Omega^-$. We assign $K\subset\RR^{n-1}$ to be the $(n-1)$-dimensional cube with center $\pi(c)$ and side length $t_K$, \begin{equation}\label{searcheq2} t_K=\lceil 4^{n-1}\beta^{-1}N\rceil^{-1}qt/4M.\end{equation} Note $K\subset \mathcal{Q}\subset I$, since $\pi(c)$ is the center of $K$, $\pi(w)\in \frac12\mathcal{Q}$, $\dist(\frac12\mathcal{Q},\mathcal{Q}^c)=qt/4$ and \begin{equation}\label{searcheq3} \dist(x,\tfrac12\mathcal{Q})\leq \frac{\diam K}{2}+|\pi(c)-\pi(w)|\ll \frac{qt}{16}+\frac{qt}{8}=\frac{3qt}{16} \quad\text{for all }x\in K.\end{equation} If $K\cap F_k=\emptyset$, we are done. Otherwise there exist points $y,z\in T_E$ such that $\pi(y)\in K$, $z\in y+\mathcal{C}$ and $s_k\leq f(z)-f(y)\leq s_{k-1}$. We claim $\pi(z)\in \mathcal{Q}$ if $C_1$ is sufficiently large. Because $z\in y+\mathcal{C}$ and $C_1h^{-1}s_k\leq t$, \begin{equation}\label{searcheq4}|\pi(z)-\pi(y)|\leq h^{-1}(f(z)-f(y))\leq h^{-1}\alpha s_k\leq t\alpha/C_1.\end{equation} Since $\pi(y)\in K$, using (\ref{searcheq3}) it follows that \begin{equation}\label{searcheq5}\dist(\pi(z),\tfrac12\mathcal{Q}) \leq |\pi(z)-\pi(y)|+\dist(\pi(y),\tfrac12\mathcal{Q})\leq t\left(\frac{\alpha}{C_1}+\frac{3q}{16}\right).\end{equation} Hence, $\dist(\pi(z),\frac12\mathcal{Q})\leq qt/4$ and $\pi(z)\in \mathcal{Q}$ provided $C_1\geq 16\alpha/q$. Assume that $C_1$ has been chosen so that this is true. Then
\begin{equation}\label{searcheq6} f(y)\leq f(z)-s_k \leq \lambda(\mathcal{Q})-s_k \leq \lambda(\tfrac12\mathcal{Q})-\frac{7}{8}s_k=f(w)-\frac{7}8s_k \leq f(c)-\frac{3}{4}s_k,\end{equation} where the last inequality holds since $|f(c)-f(w)|\leq |c-w| \leq qt/8\leq s_k/8$. Now consider the line segment $L=L(y)\subset\Omega\cap\mathcal{R}$ and let $L'$ be the vertical line segment inside $B(c,qt/8M)\subset\Omega^-\cap\mathcal{R}$ through $c$ with length $qt/4M$. By (\ref{searcheq6}), $f(L')\subset f(L)$. Hence $\pi(L)=\pi(y)$ and $\pi(L')=\pi(c)$ belong to $K$ and $f(L)\cap f(L')=f(L')$ is a line segment of length $\geq \lceil 4^{n-1}\beta^{-1}N\rceil t_K$. By Lemma \ref{Helper2}, the cube $K\subset\Lambda$. We have proved that if $C_1\geq 16\alpha/q$ and if there exists a cube $\mathcal{Q}\subset I$ of side length $qt$ such that $\lambda(\mathcal{Q})\leq \lambda(\frac12\mathcal{Q})+s_k/8$, then $I\cap(\Lambda\cup F_k^c)$ contains a cube $K$ of side length $t_K=\lceil 4^{n-1}\beta^{-1}N\rceil^{-1}qt/4M$.

To finish the lemma, set $m_0=\lceil 32^{n-1}\beta^{-1}N\rceil$ and consider the sequence of cubes $I\supset J_1\supset \dots\supset J_{m_0}$ with the same center as $I$ but with side lengths $t, t/2, \dots, t/2^{m_0}$. There are three alternatives. First if it happens $J_m\cap \pi(T)=\emptyset$ for some $1\leq m\leq m_0$, then $J_m\cap F_k=\emptyset$ and we set $J=J_m$. Otherwise we know $\lambda(J_m)$ is defined for every $m$. Second suppose that for each $m<m_0$, $\lambda(J_m)\geq\lambda(J_{m+1})+s_k/8$. Then one can find $m_0$ points $y_m\in T$  such that $\pi(y_m)\in J_m\subset I$ but $|y_m-y_{m'}|\geq |f(y_m)-f(y_{m'})|\geq s_k/8\geq t/8$ for each $m\neq m'$. The $m_0$ surface balls $\Delta(y_m,t/16)\subset \Delta(0,r)\cap \pi^{-1}(2I)$ are disjoint; by Lemma \ref{lowersurf}, \begin{equation}\begin{split} \surf(\Delta(0,r)\cap \pi^{-1}(2I))&\geq m_0\beta (t/16)^{n-1}\\ &\geq 32^{n-1}N (t/16)^{n-1}=N\surf(2I).\end{split}\end{equation} Thus, $2I\subset \Lambda$ and we can select $J=I$. Third suppose that $\lambda(J_m)\leq \lambda(J_{m+1})+s_k/8$ for some $1\leq m< m_0$. Put $C_1=2^{m_0+3}\alpha$ (which depends only on $n$, $N$, $\alpha$ and $\beta$) so that  $C_1\geq 16\alpha/q$ for $q=1/2^m$. With $\mathcal{Q}=J_m$ the argument above produces a cube  $K\subset I\cap(\Lambda\cup F_k^c)$ with side length $t_K=\lceil 4^{n-1}\beta^{-1}N\rceil^{-1}qt/4M$ and we can set $J=K$. In the worst scenario (the last case), we found a cube $J\subset I \cap (\Lambda\cup F_k^c)$ of side length  $\geq t/2^{m_0+2n+1}\beta^{-1}NM$. Therefore, since any cube contains an admissible cube at least one-quarter of its own size, we can take $C_2=2^{m_0+2n+3}\beta^{-1}NM$ (which depends only on $n$, $M$, $N$ and $\beta$).
\end{proof}

Next we iterate Lemma \ref{Helper3}. If the slope $h$ of $\mathcal{C}$ is sufficiently large, then $F_k\setminus \Lambda$ is not concentrated in any cube of size $s_k$.

\begin{lemma}\label{Helper4} For all $\delta>0$ there exists $h_1=h_1(n, C_1, C_2, \delta)$ such that \begin{equation} \surf(I\cap F_k\setminus \Lambda) \leq \delta \surf(I)\end{equation} whenever $h\geq h_1$,  $k\geq 0$ and $I\subset I_0$ is an admissible cube of side length $s_k$.
\end{lemma}

\begin{proof}Let us agree that a cube $J\subset I$ is \emph{good} if $J\subset \Lambda\cup F_k^c$; otherwise, we call $J$ \emph{bad}. Let $P$ be the smallest integer power of two that is at least $C_2$. In round one, cover $I$ by $P^{n-1}$ admissible cubes of side length $s_k/P$. If $h$ is sufficiently large, then Lemma \ref{Helper3} applies to $I$ and at least one cube of size $s_k/P$ is good and at most $P^{n-1}-1$ cubes are bad. Round two. Cover each bad cube of size $s_k/P$ by $P^{n-1}$ admissible cubes of size $s_k/P^2$. Applying Lemma \ref{Helper3} to each parent, we conclude the number of bad admissible cubes of size $s_k/P^2$ is at most $(P^{n-1}-1)^2=P^{2(n-1)}(1-P^{-n+1})^2$. Repeating this procedure through round $R$, we conclude the number of bad admissible cubes of size $s_k/P^R$ is at most $P^{R(n-1)}(1-P^{-n+1})^R$. Set $R$ to be the first positive integer such that $(1-P^{-n+1})^R<\delta$.  In order to invoke Lemma \ref{Helper3} for $R$ rounds total, we needed $s_k/P^{q}\geq C_1h^{-1}s_k$ for each $q<R$. Thus, if $h\geq h_1=C_1 P^{R-1}$, then the number of bad admissible cubes $J$ of size $s_k/P^{R}$ is at most $\delta P^{R(n-1)}$. It follows that \begin{equation}\begin{split} \surf(I\cap F_k\setminus \Lambda) &\leq \sum_{J\subset I} \surf(J\cap F_k\setminus\Lambda) \\ &\leq \delta P^{R(n-1)}(s_k/P^{R})^{n-1} = \delta \surf(I)\end{split}\end{equation} as desired.
\end{proof}

For each $k\geq 0$, call $(I_{k,j})_j$ the sequence of all admissible cubes of side length $s_k$ that meet $F_k\setminus \Lambda$. In order to control the sum in (\ref{goal3}), we associate a \emph{piece of surface} $S(I_{k,j})$ to every cube $I_{k,j}$ and then study the size and overlap of the $S(I_{k,j})$:

\begin{figure}[t]
\includegraphics[width=.35\textwidth]{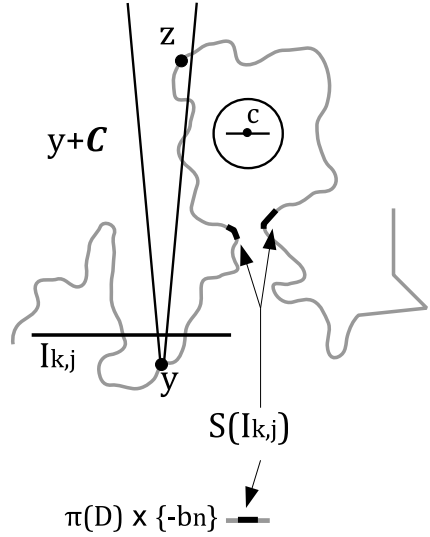}
\caption{A piece of surface $S(I_{k,j})$}
\end{figure}

Fix a large constant $\zeta\gg 1$ to be chosen later (after Lemma \ref{Helper7}). Suppose $I=I_{k,j}$ is an admissible cube of side length $s_k$ that meets $F_k\setminus \Lambda$. Let $y=y(I)$ and $z=z(I)$ be any two points of $T_E$ such that $\pi(y)\in I\cap (F_k\setminus \Lambda)$, such that $z\in y+\mathcal{C}$, and such that $s_k\leq f(z)-f(y)\leq s_{k-1}$. Let $c=c(I)$ be any non-tangential point $c=A^-(z,s_k/\zeta)$ of $z$ in $\Omega^-$. Then $|c-z|\leq s_k/\zeta$ and $B(c,s_k/\zeta M)\subset \Omega^-\cap B(z,2s_k/\zeta)$. Furthermore, $B(c,s_k/\zeta M)\subset\mathcal{R}$ if we select $\zeta \geq 8$ (compare points in the ball to $z\in T_E$). Assign $D=D(I)$ to be the $(n-1)$-dimensional disk with center $c$ and radius $s_k/2\zeta M$ that is parallel to $I_0$. We define $S=S(I)$ to be the set of all points $w$ such that (Fig. 6) \begin{enumerate}
\item $w\in \partial\Omega\cap\mathcal{R}$ or $w\in \pi(D)\times\{-b_n\}$,
\item $\pi(w)\in \pi(D)$,
\item $f(w)\leq f(c)-s_k/\zeta M$, and
\item the open vertical line segment joining $w$ to $\hat w=(\pi(w),f(c)-s_k/\zeta M)$ does not intersect $\partial\Omega$.
\end{enumerate} Including the extra $(n-1)$ disk $\pi(D)\times\{-b_n\}$ in the definition of $S$ ensures that the projection $\pi(S)=\pi(D)$ is also a disk of radius $s_k/2\zeta M$.

\begin{lemma}\label{Helper5} Assume that $h\geq 4\alpha$ and $\zeta\geq 8$. Then $\pi(S(I_{k,j}))\subset 2I_{k,j}$.\end{lemma}

\begin{proof} Let $I=I_{k,j}$ and write $y$, $z$, $c$, $D$ and $S$ for the data associated to $I$. Then $\pi(y)\in I$ and $|\pi(z)-\pi(y)|\leq h^{-1}(f(z)-f(y))\leq s_{k-1}/h=\alpha s_k/h$. Let $x\in\pi(S)$. Since the disk $\pi(S)=\pi(D)$ of radius $s_k/2\zeta M$ is centered at $\pi(c)$ and $|\pi(c)-\pi(z)|\leq s_k/\zeta$, we get \begin{equation}\label{L3-6eq1}\dist(x,I)\leq |x-\pi(c)| + |\pi(c)-\pi(z)|+|\pi(z)-\pi(y)| \leq \frac{s_k}{2\zeta M}+
\frac{s_k}{\zeta}+\frac{\alpha s_k}{h}. \end{equation} Hence, $\dist(x,I)\leq s_k/2$ and $x\in 2I$, if we require that $\zeta\geq 8$ and $h\geq 4\alpha$.
\end{proof}

Set $S_k=\bigcup_j S(I_{k,j})$ to be the union of pieces of surface $S(I_{k,j})$ associated to cubes of side length $s_k$. If the slope $h$ of $\mathcal{C}$ is sufficiently large, then the measure of $F_k\setminus\Lambda$ is controlled by the measure of $\pi(S_k)$.

\begin{lemma}\label{Helper6} There exists $h_2=h_2(n,M,\gamma,\varepsilon,C_1,C_2,\zeta)$ such that if $h\geq \max(h_2,4\alpha)$, then for each $k\geq 0$, \begin{equation}\label{Helper6a}\surf(F_k\setminus \Lambda)\leq \frac{\varepsilon}{2}\left(\gamma+\frac{1}{M\sqrt{n-1}}\right)^{-1}\surf(\pi(S_k)).
\end{equation}\end{lemma}

\begin{proof} Let $h_2=h_1(n,C_1,C_2,\delta)$ from Lemma \ref{Helper4} where \begin{equation}
\delta=\frac{\varepsilon}{2}\left(\gamma+\frac{1}{M\sqrt{n-1}}\right)^{-1}\times \frac{\omega_{n-1}}{(6\zeta M)^{n-1}}.\end{equation} If $p$ is the number of admissible cubes of size $s_k$ that meet $F_k\setminus \Lambda$, then \begin{equation}\surf(F_k\setminus \Lambda) \leq \sum_j \surf(I_{k,j}\cap F_k\setminus \Lambda) \leq p\delta s_k^{n-1}.\end{equation} Fix a cube $I=I_{k,j}$. By Lemma \ref{Helper5}, $\pi(S(I))=\pi(D(I))$ is a disk of radius $s_k/2\zeta M$ that is contained in $2I$. Hence $\surf(\pi(S_k)\cap 2I) \geq\omega_{n-1}(s_k/2\zeta M)^{n-1}$. Because the cubes $(2I_{k,j})_j$ have bounded overlap (each $x\in\RR^{n-1}$  lies in at most $3^{n-1}$ cubes), \begin{equation}p\omega_{n-1}(s_k/2\zeta M)^{n-1}\leq\sum_{j}\surf(\pi(S_k)\cap 2I_{k,j})\leq 3^{n-1}\surf(\pi(S_k)).\end{equation} (This step uses the fact that the sets $\pi(S_k)$ and $\pi(S_k)\cap 2I_{k,j}$ are measurable.) Thus, $\surf(F_k\setminus \Lambda) \leq \delta(6\zeta M)^{n-1}\surf(\pi(S_{k}))/\omega_{n-1}$, explaining our choice of $\delta$.\end{proof}

By choosing good parameters, we can make the pieces of surface $S_0, S_1, S_2, \dots$ disjoint!

\begin{lemma}\label{Helper7} We can find $\alpha\geq 4M\sqrt{n-1}$ depending only on $n$ and $M$, and find $\zeta\geq 8$ depending only on $n$, $N$ and $\beta$ such that $\overline{S_k}\cap \overline{S_{k'}}=\emptyset$ for all $k\neq k'$ whenever $h\geq \alpha\zeta$.\end{lemma}

\begin{proof} To start assume that $h\geq 4\alpha$ and $\zeta\geq 8$. Then the pieces of surface $(S(I_{k,j}))_j$ have finite overlap by Lemma \ref{Helper5}. Hence $\overline{S_k} = \bigcup_j\overline{S(I_{k,j})}$ for each $k$.

Let cubes $I=I_{k,j}$ and $I'=I_{k',j'}$ be given. We shall write $y$, $z$, $c$, $D$ and $S$ for the data associated to $I$ and write $y'$, $z'$, $c'$, $D'$ and $S'$ for the data associated to $I'$. Suppose to get a contradiction that $k\geq k'+1$ and $\overline{S}\cap \overline{S'}\neq\emptyset$. Then there exist $w\in S$ and $w'\in S'$ such that \begin{equation} |w-w'| \leq \frac{s_{k'}}{4\zeta M}.\end{equation} Let $L_w$ denote the vertical line segment joining $w$ to $\hat w=(\pi(w),f(c)-s_k/\zeta M)$, and write $L_y=L(y)\subset\Omega$ for the vertical line segment over $y$ in $\mathcal{T}$. Also set $B'=B(c',s_{k'}/\zeta M)\subset\Omega^-$.

Our first claim is $\pi(y)\in \pi(B')$ and $f(y)>f(c')$ for certain choices $h$, $\alpha$ and $\zeta$. Indeed, since $\pi(w)\in \pi(S)$, \begin{equation}\begin{split}\label{L3-6a}|\pi(w)-\pi(y)| &\leq |\pi(w)-\pi(c)|+|\pi(c)-\pi(z)|+|\pi(z)-\pi(y)|\\
&\leq \frac{s_k}{2\zeta M}+\frac{s_k}{\zeta}+\frac{\alpha s_k}{h}.\end{split}\end{equation} If we select $h\geq \alpha\zeta$, then (\ref{L3-6a}) implies that $|\pi(w)-\pi(y)|\leq 3s_k/\zeta$. Hence, \begin{equation}\begin{split} |\pi(y)-\pi(c')| &\leq |\pi(y)-\pi(w)|+|\pi(w)-\pi(w')|+|\pi(w')-\pi(c')|\\ &\leq \frac{3s_k}\zeta + \frac{s_{k'}}{4\zeta M} + \frac{s_{k'}}{2\zeta M} \leq \frac{3s_{k'}}{\alpha\zeta} + \frac{s_{k'}}{4\zeta M} + \frac{s_{k'}}{2\zeta M},\end{split}\end{equation} where $s_k\leq s_{k'}/\alpha$ since $k\geq k'+1$. Thus, $|\pi(y)-\pi(c')|\leq s_{k'}/\zeta M$ and $\pi(y)\in \pi(B')$, provided $\alpha \geq 12M$. If $f(y)\leq f(c')$, then $\pi(y)\in \pi(B')$ implies that $L_y\subset\Omega$ intersects $B'\cap f^{-1}(c')\subset \Omega^-$, which is absurd. Therefore, $f(y)>f(c')$, as claimed.

Next we claim $L_w\subset \Omega^-$. On one hand the upper endpoint $\hat w$ of $L_w$ satisfies \begin{equation}f(\hat w)\geq f(z)-\frac{2s_k}{\zeta}\geq f(z)-\frac{s_k}{4}
\geq f(y)+\frac{3s_k}{4}>f(c')\end{equation} since $\zeta\geq 8$. On the other hand the lower endpoint $w$ of $L_w$ satisfies, \begin{equation}f(w)\leq f(w')+\frac{s_{k'}}{4\zeta M}\leq f(c')-\frac{3s_{k'}}{4\zeta M} < f(c').\end{equation} Thus, since $\pi(w)\in\pi(B')$ (in fact $|\pi(w)-\pi(c')|\leq (3/4)s_{k'}/{\zeta M}$), the line segment $L_w$ intersects $B'\cap f^{-1}(c')\subset \Omega^-$. But $L_w$ does not intersect $\partial\Omega$ (by definition of $S$), so $L_w\subset\Omega^-$.

Finally, since $f(\hat w)\geq f(y)+3s_k/4$ and $f(y)>f(c')>f(w)$, we know  the interval $f(L_w)\cap f(L_y)$ has length at least $3s_k/4$. Previously we showed $|\pi(w)-\pi(y)|\leq 3s_k/\zeta$ if $h\geq \alpha\zeta$. Thus, $L_w$ and $L_y$ lie over a cube $J$ of side length $\lceil 4^{n-1} \beta^{-1} N\rceil^{-1} 3s_k/4$ if  $\zeta\geq 4\lceil 4^{n-1} \beta^{-1} N\rceil$. By Lemma \ref{Helper2}, $\pi(y)=\pi(L_y)\in J\subset\Lambda$. This contradicts the fact $\pi(y)\not\in \Lambda$ (by the definition of $y$). Examining conditions on the parameters assumed above reveals  the lemma holds with $\zeta=4\lceil 4^{n-1}\beta^{-1} N\rceil$,
$\alpha =\max(4M\sqrt{n-1},12M)$ and $h\geq \alpha\zeta$.
\end{proof}

We are ready to conclude. Use Lemma \ref{Helper7} to pick the constants $\alpha$ and $\zeta$, and set $h=\max(h_0,h_2,\alpha\zeta)$. Then the pieces of surface $\overline{S_k}\subset \mathcal{S}:=\Delta(Q,r)\cup (2I_0\times\{-b_n\})$ are disjoint and measurable. Thus, since $\surf(\Delta(Q,r))\leq\gamma r^{n-1}$, \begin{equation}\sum_{k=0}^{\infty}\surf(\pi(S_k))\leq \sum_{k=0}^\infty \surf\left(\overline{S_k}\right)\leq\surf(\mathcal{S})\leq \left(\gamma+\frac{1}{M\sqrt{n-1}}\right)r^{n-1}.\end{equation} Using Lemma \ref{Helper6}, we conclude
\begin{equation}\sum_{k=0}^\infty \surf(F_k\setminus \Lambda)\leq \frac{\varepsilon}{2}\left(\gamma+\frac{1}{M\sqrt{n-1}}\right)^{-1} \sum_{k=0}^\infty \surf(\pi(S_k))\leq \frac{\varepsilon}{2}r^{n-1}.\end{equation} Therefore,  (\ref{goal3}) holds and Proposition \ref{PropG} is established.

\section{Harmonic Measure on NTA Domains}
\label{section4}

In this section we prove Theorem \ref{macthm} on the absolute continuity of harmonic measure. An NTA domain is a corkscrew domain (studied in \S\ref{section2} above) that also admits a Harnack chain condition. The class of NTA domains was introduced by Jerison and Kenig \cite{JK}. Given $X_1,X_2\in\Omega$ a \emph{Harnack chain} from $X_1$ to $X_2$ is a sequence of open balls in $\Omega$ such that the first ball contains $X_1$, the last ball contains $X_2$, and consecutive balls intersect.

\begin{definition}A connected open set $\Omega\subset\RR^n$ satisfies the \emph{Harnack chain condition} with constants $M>1$ and $R>0$ if for every $Q\in\partial\Omega$ and $0<r<R$ when a pair of points $X_1,X_2\in \Omega\cap B(Q,r)$ satisfy \begin{equation} \min_{j=1,2}\dist(X_j,\partial\Omega)>\varepsilon\quad\text{and}\quad|X_1-X_2|< 2^k\varepsilon\end{equation} then there exists a Harnack chain from $X_1$ to $X_2$ of length $Mk$ such that the diameter of each ball is bounded below by $M^{-1}\min_{j=1,2}\dist(X_j,\partial\Omega)$.\end{definition}

\begin{definition}\label{NTA} A domain $\Omega\subset\RR^n$ is \emph{non-tangentially accessible} or \emph{NTA} if there exist $M>1$ and $R>0$ such that (i) $\Omega$ satisfies the corkscrew and Harnack chain conditions, (ii) $\RR^n\setminus\Omega$ satisfies the corkscrew condition.\end{definition}

The exterior corkscrew condition guarantees an NTA domain $\Omega\subset\RR^n$ is regular for the Dirichlet problem; i.e.\ for every $f\in C_c(\partial\Omega)$ there exists $u\in C^2(\Omega)\cap C(\overline{\Omega})$ such that $\Delta u=0$ in $\Omega$ and $u=f$ on $\partial\Omega$. Together the maximum principle and Riesz representation theorem yield a family of Borel regular probability measures $\{\omega^X\}_{X\in\Omega}$ on $\partial\Omega$ such that \begin{equation}u(X)=\int_{\partial\Omega} f(Q)d\omega^X(Q)\end{equation} is the unique harmonic extension of $f\in C_c(\partial\Omega)$. We call $\omega^X$ the \emph{harmonic measure} of $\Omega$ with \emph{pole} at $X$. Because $\omega^{X_1}\ll \omega^{X_2}\ll\omega^{X_1}$ for any $X_1,X_2\in\Omega$ (by Harnack's inequality), it makes sense to discuss null sets of harmonic measure $\omega=\omega^{X_0}$ with respect to some fixed pole $X_0\in\Omega$ far away from the boundary.

The special feature of harmonic measure on NTA domains (versus corkscrew domains) that we need below is the the following localization property.

\begin{lemma}[\cite{JK} (4.18)] \label{localize} There exists $C=C(n,M)>0$ with the following property. Let $\Omega\subset\RR^n$ be NTA with constants $M>1$ and $R>0$.  Assume the pole of harmonic measure $\omega=\omega^{X_0}$ satisfies $X_0\in \Omega\setminus B(Q,2r)$ for some $Q\in \partial\Omega$ and $r<R/2$. Then for every non-tangential point $a=A^+(Q,r)$ and every Borel set $E\subset\Delta(Q,r)$, \begin{equation} C^{-1} \omega^{a}(E) \leq \frac{\omega(E)}{\omega(\Delta(Q,r))} \leq C \omega^{a}(E).\end{equation}\end{lemma}

\begin{remark}\label{NTAremark} In Definition \ref{NTA} we allow an NTA domain $\Omega\subset\RR^n$ to be either bounded or unbounded. The proof of Lemma \ref{localize} for bounded domains in \cite{JK} carries through to the unbounded case without modification; c.f. \cite{KT}.
\end{remark}

At every boundary point of finite lower density there is a shrinking sequence of scales on which the harmonic measure and the surface measure are comparable in the sense of (\ref{Lsceq}). The proof of Proposition \ref{shrinking} below follows the same structure of David and Jerison's proof of Theorem 2 in \cite{DJ}; however, we keep careful track of the constants appearing from Lipschitz approximations of the domain (Theorem \ref{approxthm} and Lemma \ref{etalemma}). The required technical tools are the localization principle for harmonic measure (Lemma \ref{localize}) and the maximum principle for harmonic functions.

\begin{proposition}\label{shrinking} There exist constants $0<\delta<1$ and $0<\varepsilon<1$ depending only on $n$, $M$ and $\gamma$ with the following property. Let $\Omega\subset\RR^n$ be NTA with constants $M>1$ and $R>0$. If $\liminf_{r\downarrow 0} \surf(\Delta(Q,r))/r^{n-1}<\gamma<\infty$, then there is a sequence of numbers $0<r_i<R$ such that $\lim_{i\rightarrow\infty} r_i=0$ and for every Borel set $E\subset\Delta(Q,r_i)$: \begin{align}
\label{ainfty1}\omega(E)\leq \delta\omega(\Delta(Q,r_i))&\Rightarrow \surf(E)\leq\varepsilon\surf(\Delta(Q,r_i)),\\
\label{ainfty2}\surf(E)\leq \delta\surf(\Delta(Q,r_i))&\Rightarrow \omega(E)\leq \varepsilon\omega(\Delta(Q,r_i)).\end{align}\end{proposition}

\begin{proof} Let $\Omega\subset\RR^n$ be NTA with constants $M>1$ and $R>0$ and assume $Q\in\partial\Omega$ satisfies $\liminf_{r\downarrow 0} \surf(\Delta(Q,r))/r^{n-1}<\gamma<\infty$. Then there exists a sequence of numbers $0<r_i<R$ decreasing to zero such that $\surf(\Delta(Q,r_i))\leq \gamma r_i^{n-1}$. Let $\psi=\psi(n,M)$ and $\eta=\eta(n,M,\gamma)$ be the constants given by Theorem \ref{approxthm} and Lemma \ref{etalemma}. By passing to a subsequence of $r_i$ if necessary, we may assume that the pole of $\omega$ lies outside $B(Q,2r_i)$ for all $i$ (so that we can invoke Lemma \ref{localize}).

Fix $r_i$ and pick any non-tangential point $a=A^+(Q,r_i/2)$ of $\Omega$. Let $\omega^a$ denote harmonic measure of $\Omega$ with pole at $a$ and let $\Delta=\Delta(Q,r_i)$. By Lemma \ref{localize}, \begin{equation}\label{shrink1} C^{-1}\omega^a(E)\leq \frac{\omega(E)}{\omega(\Delta)}\leq C\omega^a(E)\quad\text{for every Borel set }E\subset\Delta,\end{equation} where the constant $C>1$ only depends on the dimension and NTA constants of $\Omega$. By Theorem \ref{approxthm} there is a Lipschitz domain $\Omega_L\subset\RR^n$ such that (i) $a\in\Omega_L\subset\Omega\cap B(Q,r_i)$ and (ii) $F=\partial\Omega_L\cap\partial\Omega$ satisfies $\surf(F)\geq \psi r_i^{n-1}$. Let $\omega^a_L$ denote the harmonic measure of $\Omega_L$ with pole at $a$.

Assume $E\subset\Delta$ is Borel and $\omega(E)<\delta \omega(\Delta)$. By (\ref{shrink1}) and the maximum principle, \begin{equation}\label{shrink2} \omega_L^a(E\cap F)\leq \omega^a(E\cap F)\leq \omega^a(E)\leq C\delta.\end{equation} If $C\delta\leq\eta$, then (\ref{etainfty}) and (\ref{shrink2}) imply $\surf(E\cap F)\leq (\psi/2)r_i^{n-1}$. Hence \begin{equation}\begin{split} \surf(F\setminus E)&=\surf(F)-\surf(E\cap F)\\&\geq \psi r_i^{n-1}-(\psi/2)r_i^{n-1}=(\psi/2)r_i^{n-1}.\end{split}\end{equation} It follows that \begin{equation}\begin{split} \surf(E) &= \surf(\Delta)-\surf(\Delta\setminus E)\leq \surf(\Delta)-\surf(F\setminus E)\\
&\leq \surf(\Delta)-\frac\psi2r_i^{n-1}
=\surf(\Delta)-\frac{\psi}{2\gamma}\gamma r_i^{n-1} \\ &\leq \surf(\Delta)-\frac{\psi}{2\gamma}\surf(\Delta)=\left(1-\frac{\psi}{2\gamma}\right)\surf(\Delta).\end{split}
\end{equation} Thus (\ref{ainfty1}) holds for all $0<\delta\leq \eta/C$ and for all $1-\psi/2\gamma\leq\varepsilon<1$.

Now assume $E\subset\Delta$ satisfies $\surf(E)\leq\delta\surf(\Delta)\leq \delta\gamma r_i^{n-1}$. If $\delta\gamma\leq\psi/4$, then $\surf(F\setminus E)=\surf(F)-\surf(E\cap F) \geq (3\psi/4)r_i^{n-1}$. The contrapositive of (\ref{etainfty}) implies $\omega_L^a(F\setminus E)> \eta$. By (\ref{shrink1}) and the maximum principle, $\omega(F\setminus E)\geq(\eta/C)\omega(\Delta)$. We conclude \begin{equation}\omega(E)=\omega(\Delta)-\omega(\Delta\setminus E)
\leq \omega(\Delta)-\omega(F\setminus E)\leq (1-\eta/C)\omega(\Delta).\end{equation} Thus (\ref{ainfty2}) holds for all $0<\delta\leq \psi/4\gamma$ and for all $1-\eta/C\leq \varepsilon<1$. Therefore, \begin{equation}\delta=\min\{\eta/C,\psi/4\gamma\}\quad\text{and}\quad \varepsilon=1-\delta\end{equation} which depend only on $n$, $M$ and $\gamma$ suffice.\end{proof}

To stich together estimates in Proposition \ref{shrinking} at different points, we use a Vitali type covering lemma for Radon measures in $\RR^n$.

\begin{theorem}[\cite{M} Theorem 2.8] \label{vitali} Let $\mu$ be a Radon measure on $\RR^n$,  $A\subset\RR^n$ and $\mathcal{B}$ a family of closed balls such that each point of $A$ is the center of arbitrarily small balls; i.e., \begin{equation} \label{finecover} \inf\{r:B(x,r)\in \mathcal{B}\}=0\quad\text{for all }x\in A.\end{equation} Then there are disjoint balls $B_i\in\mathcal{B}$ such that \begin{equation}\mu\left(A\setminus\bigcup_i B_i\right)=0.\end{equation}
\end{theorem}

We now establish the main theorem. Recall: \emph{Let $\Omega\subset\RR^n$ be NTA. Then the set \begin{equation}A=\left\{Q\in\partial\Omega:\liminf_{r\downarrow 0}\frac{\surf(\Delta(Q,r))}{r^{n-1}}<\infty\right\}\end{equation} is $(n-1)$-rectifiable and $\omega\res A\ll\sigma\res A\ll\omega\res A$.}

\begin{proof}[Proof of Theorem 1.2]
The set $A$ is $(n-1)$-rectifiable by Corollary \ref{corkrect}. Define \begin{equation}A_k=\left\{Q\in\partial\Omega:\liminf_{r\downarrow 0} \frac{\surf(\Delta(Q,r)}{r^{n-1}}< k\right\}\end{equation} for each integer $k\geq 1$. Then $A=\bigcup_{k=1}^\infty A_k$ and to show $\omega\res A\ll\sigma\res A\ll\omega\res A$ we may prove that $\omega(E)=0$ if and only if $\surf(E)=0$ for every $k\geq 1$ and compact set $E\subset A_k$. (It is enough to take $E$ compact, because $\surf \res A_k$ is Radon.)

Let $E\subset A_k$ be an arbitrary compact set and assume $\omega(E)=0$. By Proposition \ref{shrinking} there exist constants $0<\delta<1$ and $0<\varepsilon<1$ such that (\ref{ainfty1}) and (\ref{ainfty2}) hold for all $Q\in E$ (along some sequence $r_i\downarrow 0$ depending on $Q$). Let $U$ be any (relatively) open set $U\subset\partial\Omega$ such that $E\subset U$. The family $\mathcal{F}=\{\Delta(Q,r_i)\}$ of all closed balls with center $Q\in E$ and radii $r_i$ satisfying $\Delta(Q,r_i)\subset U$ and (\ref{ainfty1}) is a fine cover of $E$, in the sense of (\ref{finecover}). By Theorem \ref{vitali}, there exists a disjoint sequence $\Delta_i$ of disks in $\mathcal{F}$ such that \begin{equation}\label{MT-1} \surf\left(E\setminus\bigcup_i\Delta_i\right)=0.\end{equation} Since $\omega (E\cap \Delta_i)=0\leq\delta \omega(\Delta_i)$ for each $i$, $\surf(E\cap\Delta_i) \leq \varepsilon\surf(\Delta_i)$ for each $i$ by (\ref{ainfty1}). Thus, by (\ref{MT-1}),
\begin{equation}\label{MT-2} \surf(E) = \sum_{i} \surf(E\cap\Delta_i)\leq\varepsilon\sum_i\surf(\Delta_i)
\leq \varepsilon\surf(U).\end{equation} Because $U\supset E$ was an arbitrary open set, by the outer regularity of Radon measures, $\surf(E)\leq \varepsilon\surf(E)$. But $\surf(E)<\infty$ (since $E$ is compact) and $0<\varepsilon<1$. Therefore, $\surf(E)=0$ whenever $\omega(E)=0$.

If $E\subset A_k$ is a compact set such that $\surf(E)=0$, then the same argument with the roles of $\omega$ and $\surf$ reversed and $(\ref{ainfty2})$ in place of $(\ref{ainfty1})$ shows $\omega(E)=0$. This completes the proof of absolute continuity on $A$.
\end{proof}

\section{Hausdorff Dimension and Wolff Snowflakes}
\label{section5}

We now present two corollaries of Theorem \ref{macthm} related to the dimension of harmonic measure. Let $\dim E$ denote the Hausdorff dimension of a set $E\subset\RR^n$. The \emph{(upper) Hausdorff dimension} of harmonic measure, \begin{equation}\Hdim\omega=\inf\{\dim E:E\subset\partial\Omega\text{ is Borel and }\omega(E)=1\},\end{equation} is the smallest dimension of a set with full harmonic measure. In \cite{Makarov} Makarov showed that $\Hdim\omega=1$ for every simply connected planar domain (independent of the Hausdorff dimension of the boundary); moreover, $\omega\ll \mathcal{H}^s$ for all $s<1$ and $\omega\perp\mathcal{H}^t$ for all $t>1$. Higher dimensions display different behavior.

Wolff \cite{W} constructed NTA domains $\Omega\subset\RR^3$ such that $\Hdim\omega>2$ and other NTA domains $\Omega\subset\RR^3$ such that $\Hdim\omega<2$.
Extending this construction, Lewis, Verchota and Vogel \cite{LVV} built 2-sided NTA domains $\Omega\subset\RR^n$ (i.e.\ $\Omega^+=\Omega$ and $\Omega^-=\RR^n\setminus\overline{\Omega}$ are both NTA) called \emph{Wolff snowflakes} such that\begin{enumerate}
\item $\Hdim\omega^+>n-1$ and $\Hdim\omega^->n-1$,
\item $\Hdim\omega^+>n-1$ and $\Hdim\omega^-<n-1$,
\item $\Hdim\omega^+<n-1$ and $\Hdim\omega^-<n-1$.
\end{enumerate} Here $\omega^+$ denotes harmonic measure on the interior $\Omega^+$ and $\omega^-$ denotes harmonic measure on the exterior $\Omega^-$ of $\Omega$. While surface measure $\sigma=\surf\res\partial\Omega$ is clearly infinite for Wolff snowflakes of type (1) or (2), the same is not apparent for snowflakes of type (3). This is the first application of Theorem \ref{macthm}: every Wolff snowflake has infinite surface measure. In fact, if the dimension of harmonic measure on a (1-sided) NTA domain is small, then the surface measure is infinite at all locations and scales.

\begin{theorem}Let $\Omega\subset\RR^n$ be NTA. If $\Hdim\omega<n-1$, then $\sigma$ is locally infinite; i.e., $\surf(\Delta(Q,r))=\infty$ for every $Q\in\partial\Omega$ and $r>0$.\end{theorem}

\begin{proof} Assume that $\Omega\subset\RR^n$ is NTA and $\Hdim\omega<n-1$. Then there exists a Borel set $E\subset\partial\Omega$ such that $\dim E<n-1$ and $\omega(\partial\Omega\setminus E)=0$. Suppose for contradiction that $\surf(\Delta(Q,r))<\infty$ for some $Q\in\partial\Omega$ and $r>0$. Then harmonic measure and surface measure have the same null sets on $\Delta(Q,r)\cap A$ by Theorem \ref{macthm}. On one hand, \begin{equation}\omega(\Delta(Q,r)\cap A)=\omega(E\cap\Delta(Q,r)\cap A)=0\end{equation} since $\surf(E)=0$. On the other hand, \begin{equation}\omega(\Delta(Q,r)\cap A)>0\end{equation} since $\surf(\Delta(Q,r)\cap A)=\surf(\Delta(Q,r))>0$ by Lemma \ref{lowersurf}. The fallacy is clear. We conclude $\surf(\Delta(Q,r))=\infty$ for every $Q\in\partial\Omega$ and $r>0$.
\end{proof}

In \cite{KPT} Kenig, Preiss and Toro used the tangent measures of harmonic measure to demonstrate $\Hdim\omega^+=\Hdim\omega^-=n-1$ on every 2-sided NTA domain $\Omega\subset\RR^n$ ($n\geq 3$) with $\omega^+\ll\omega^-\ll\omega^+$. (Thus the interior and exterior harmonic measures on Wolff snowflakes are mutually singular.) Using Theorem \ref{FMnta}, we obtain a similar result for (1-sided) NTA domains of locally finite perimeter.

\begin{theorem} \label{dimthm} Let $\Omega\subset\RR^n$ be NTA. If $\surf\res\partial\Omega$ is Radon, then $\Hdim\omega=n-1$.\end{theorem}

\begin{proof} If $E\subset\partial\Omega$ is Borel and has dimension $t<n-1$, then $\surf(\partial\Omega\setminus E)>0$. Hence, $\omega(\partial\Omega\setminus E)>0$, because $\sigma\ll\omega$ by Theorem \ref{FMnta}. Since no set $E\subset \partial\Omega$ of Hausdorff dimension $t<n-1$ has full harmonic measure, we get $\Hdim\omega\geq n-1$. Conversely, $\Hdim\omega\leq \dim\partial\Omega=n-1$, since  $\surf\res\partial\Omega$ is Radon.\end{proof}

\end{document}